\documentclass[10pt,a4paper,reqno]{amsart}

\usepackage{amsmath,amssymb,amsthm,amsfonts}
\usepackage{amsbsy}
\usepackage[utf8]{inputenc}
\usepackage[symbol]{footmisc}

\usepackage[english]{babel}
\usepackage{url}
\usepackage{wrapfig}
\usepackage{enumitem}
\usepackage{multicol}
\usepackage{moreenum}
\usepackage{subcaption}
\usepackage{xspace}

\usepackage{wrapfig}

\usepackage{tikz-cd}

\usepackage[a4paper,left=25mm,right=25mm,top=30mm,bottom=30mm]{geometry}

\usepackage{xcolor}

\definecolor{LinkColor}{rgb}{0,0,0} 
\usepackage[colorlinks=true,linkcolor=LinkColor,citecolor=LinkColor,urlcolor= LinkColor, naturalnames, hyperindex, pdfstartview=FitH, bookmarksnumbered, plainpages]{hyperref}

\newcounter{dummy}
\makeatletter
\providecommand{\@LN}[2]{}
\newcommand\myitem[1][]{\item[#1]\refstepcounter{dummy}\def\@currentlabel{#1}}
\makeatother

\newtheorem{theorem}{Theorem}[section]
\newtheorem{corollary}[theorem]{Corollary}
\newtheorem{lemma}[theorem]{Lemma}
\newtheorem{proposition}[theorem]{Proposition}

\newtheorem{maintheorem}{Theorem}
\newtheorem{mainproposition}[maintheorem]{Proposition}

\theoremstyle{definition}
\newtheorem{definition}[theorem]{Definition}
\newtheorem{remark}[theorem]{Remark}
\newtheorem{example}[theorem]{Example}
\newcommand{\cut}{\textsf{cut}\xspace}
\newcommand{\SL}{\operatorname{SL}}
\newcommand{\GL}{\operatorname{GL}}

\newcommand{\rank}{\operatorname{rank}}

\newcommand{\Z}{\textup{Z}}
\newcommand{\U}{\textup{U}}
\newcommand{\V}{\textup{V}}

\newcommand{\ZZ}{\mathbb{Z}}
\newcommand{\QQ}{\mathbb{Q}}
\renewcommand{\O}{\mathcal{O}}

\usepackage[normalem]{ulem}

\title{Abelianization of the unit group of an integral group ring}

\author[A.~B\"achle]{Andreas B\"achle}
\address{(Andreas Bächle)}
 \email{\href{mailto:ABaechle@gmx.net}{ABaechle@gmx.net}}

\author[S. Maheshwary]{Sugandha Maheshwary}
\address{(Sugandha Maheshwary) Indian Institute of Science Education and Research, Mohali, Sector 81, Mohali (Punjab)-140306, India.}
\email{\href{mailto:sugandha@iisermohali.ac.in}{sugandha@iisermohali.ac.in}}

\author[L.~Margolis]{Leo Margolis}
\address{(Leo Margolis) Vakgroep Wiskunde, Vrije Universiteit Brussel, Pleinlaan 2, 1050 Brussels, Belgium.}
\email{\href{mailto:leo.margolis@vub.be}{leo.margolis@vub.be}}


\thanks{The work of the first and third author on this project was supported by the Research Foundation Flanders (FWO - Vlaanderen).
 The research of the second author is supported by DST (Department of Science and Technology), India (INSPIRE/04/2017/000897)}

\begin{document}

\maketitle

\begin{abstract}	For a finite group $G$ and $U: = \U(\mathbb{Z}G)$, the group of units of the integral group ring of $G$, we study the implications of the structure of $G$ on the abelianization $U/U'$ of $U$. We pose questions on the connections between the exponent of $G/G'$ and the exponent of $U/U'$ as well as between the ranks of the torsion-free parts of $\Z(U)$, the center of $U$, and $U/U'$. We show that the units originating from known generic constructions of units in $\mathbb{Z}G$ are well-behaved under the projection from $U$ to $U/U'$ and that our questions have a positive answer for many examples. We then exhibit an explicit example which shows that the general statement on the torsion-free part does not hold, which also answers questions from \cite{BJJKT1}. 
\end{abstract}

\section{Introduction}

A classical theorem attributed to I.~Schur asserts that if the center $\Z(H)$ of a group $H$ has finite index in $H$, then $H$ has a finite commutator subgroup $H'$. Several variants of Schur's result have been studied. For example, B. H. Neumann proved that if $H$ is a finitely generated group with finite commutator subgroup, then $H/ \Z(H)$ is finite \cite{Neu51}. In this article, we are interested in the following variations of Schur's theorem for certain classes of groups. Let $H$ be a group.
\begin{enumerate}
	\item\label{var1} Does $[H:H'] < \infty$ imply $|\Z(H)| < \infty$?
	\item\label{var2} Does $|\Z(H)| < \infty$ imply $[H:H'] < \infty$?
\end{enumerate}

These assertions are not true in general: For example, by \cite[Corollary~10.2]{Swan}, the group $\SL_2(\mathbb{Z}[\sqrt{-2}])$ has infinite abelianization, yet it has finite center. In fact, for some interesting classes of groups, a positive answer implies further structural results as we explain below.
We provide more details and positive as well as negative examples for \eqref{var1} and \eqref{var2} in Section~\ref{sec:prelim}.\\

The abelianization and the center of the unit group of an order in a finite-dimensional semi-simple \linebreak rational algebra are closely related through $K$-theory. For instance, if $\O$ is an order in a finite-dimensional semi-simple rational algebra with unit group $U = \U(\O)$, then $\rank  U/U'\geqslant \rank K_1(\O)=\rank \Z(U)$, where $K_1(\O)=\GL(\O)/\GL(\O)'$, and $\rank A$ denotes the torsion-free rank of a finitely generated abelian group $A$ (see e.g.\ \cite[Proposition~6.1]{BJJKT1}). \\

In this article, our main object of interest is the unit group $\U(\ZZ G)$ of the integral group ring $\ZZ G$ of a finite group $G$. In particular, as $\ZZ G$ is an order in the rational group algebra $\QQ G$, we see that \eqref{var1} holds for the group $\U(\ZZ G)$. 
As we are not interested in the influence of the coefficient ring $\ZZ$ on the structure of $\U(\ZZ G)$, we define the group $\V(\ZZ G)$ of normalized units as the subgroup formed by elements of $\U(\ZZ G)$ of augmentation $1$. Note that $\U(\ZZ G)/\U(\ZZ G)' \cong C_2 \times \V(\ZZ G)/\V(\ZZ G)'$. 

From now onwards, let $G$ be a finite group and $V: = \V(\ZZ G)$. \\

	The following variations of \eqref{var2} appear natural for $V$:
	\vspace{-0.4cm}
	\begin{enumerate}
		\item[] 
		\myitem[(R1)]\label{R1} Is $\rank V/V' = \rank \Z(V)$?
		\myitem[(R2)]\label{R2} Assume $\Z(V)$ is finite. Is $V/V'$ also finite?\\			
\end{enumerate}

The above questions address the free subgroup of $V/V'$, while the ones listed next try to get hold of its torsion elements. Denote by $\exp H$, the exponent of a group $H$, i.e., the greatest common divisor of the orders of torsion elements of $H$. 

\begin{enumerate}
	\item[] 
	\myitem[(E1)]\label{E1} Is $\exp V/V' = \exp G/G'$? 
	\myitem[(E2)]\label{E2} Does $\exp V/V'$ divide $\exp G$?
	\myitem[(P)]\label{P} If $V/V'$ contains an element of order $p$, does $G$ contain an element of order $p$, for every prime $p$? \\
\end{enumerate}

As already observed, if $V$ has finite abelianization, then the center of $V$ is finite. Also, by the Berman-Higman theorem (see e.g.\,\cite[Proposition 1.5.1]{JdR1}), we have that every central unit of finite order in $\ZZ G$ is trivial, i.e., it is an element of $\Z(G)$. Consequently, it follows that if $V$ has finite abelianization, then the center of $V$ consists only of trivial units and in this case, $G$ is known as a \cut group - \textbf{c}entral \textbf{u}nits \textbf{t}rivial \cite{BMP17}. This class of groups has been intensively studied lately \cite{BMP17,Mah18, Bac18, MP18, BCJM, Tre19, BBM20, Gri20}. From yet another perspective, understanding $V/V'$ can also be regarded as a first important step towards understanding the lower central series of $V$, an important object for which very few results are available \cite{Mah21}.\\

Surprisingly, little is known about the structure of $V/V'$. Some examples have been studied in \cite{SGV97, SG00, SG01}, but the only general result we are aware of is that \ref{R1} has a positive answer in case $\QQ G$ has no exceptional component (see Definition~\ref{exc_comp}) \cite[Theorem~6.3]{JOdRVG}. 
A crucial role in the proof of this result was played by the so-called Bass units. The second important generic construction of units in integral group rings are the so-called bicyclic units (see Section~\ref{sec:BicAndBass} for definitions). Using more constructive methods, we generalize \cite[Theorem~6.3]{JOdRVG}:

\begin{maintheorem}\label{theo:ranks}  Let $G$ be a finite group and let $\mathcal{B}$ be the subgroup of $V = \V(\ZZ G)$, generated by the elements of $G$, the bicyclic and the Bass units of $\ZZ G$. If $\mathcal{B}$ has finite index in $V$, then $\rank V/V' = \rank \Z(V)$, i.e., \ref{R1} has a positive answer.
\end{maintheorem}

In addition, we show that the torsion-free contribution by $\mathcal{B}$ to $V/V'$ is completely under control and that the torsion part satisfies \ref{E2}.

\begin{mainproposition}\label{prop:MainBicBass}
Let $G$ be a finite group and let $\mathcal{B}$ the subgroup of $V = \V(\ZZ G)$ generated by the elements of $G$, the bicyclic and the Bass units of $\ZZ G$. Denote by $\varphi \colon V \to V/V'$ the natural projection. Then $\rank \varphi(\mathcal{B}) = \rank \Z(V)$ and $\exp \varphi(\mathcal{B})$ divides $\exp G$. \end{mainproposition}

	The condition for the group $\mathcal{B}$ as in Theorem~\ref{theo:ranks} to have finite index in $V$ is a question of major interest and has been verified for many classes of groups, see \cite[Chapters~11 and 12]{JdR1} and \cite[Chapter~3]{Sehgal1993}.  Amongst other classes, it is known to hold for dihedral groups, which in lieu of Theorem \ref{theo:ranks} yields that \ref{R1} is true for these groups (Corollary \ref{cor:dihedral}). 
	Consequently, \ref{R1} and hence \ref{R2} may also be true for a group $G$, even if  $\QQ G$ has an exceptional component.
	
	As supporting evidence for \ref{E1}, \ref{E2} and \ref{P}, we give a positive answer for the strongest of these assertions for the groups in Theorems \ref{prop:Dihedral} and \ref{theo:SmallOrders}.
	
	\begin{maintheorem}\label{prop:Dihedral} Let $G$ be a dihedral group of order $2p$, where $p$ is an odd prime, and let $V = \V(\ZZ G)$. Then $\exp V/V' =\exp G/G'$, i.e., \ref{E1} holds for $G$.
 \end{maintheorem}

Abelianizations are not only interesting in their own right, but they are also closely related to geometric group theory, for instance, to fixed point properties of group actions (see e.g.\ \cite[I.6.1 Theorem~15]{Ser80}). Question \ref{R2} is exactly Question \eqref{var2} in the situation of units in integral group rings. It was recently raised explicitly in the context of the study of fixed point properties of $V$ \cite{BJJKT1}. A positive answer to \ref{R2} for $G$ proves a trichotomy \cite[Question~7.8, Proposition~7.9]{BJJKT1} which provides a very clear picture of these fixed point properties for the unit group of an integral group ring. This question was the starting point for this investigation. We provide a negative answer to \ref{R2} disproving the trichotomy in general (Example \ref{ex:Ord16}(ii)). We also answer two more questions from \cite{BJJKT1} (Corollary \ref{cor:D16+}).

\begin{maintheorem}\label{theo:SmallOrders} Let $G$ be a group and let $V = \V(\ZZ G)$.
\begin{enumerate}
\item If $G$ is of order at most $15$, then \ref{R1} and \ref{E1} have  positive answers for $G$.
\item There are non-abelian groups of order $16$ for which \ref{R1} has a positive answer. There is a group of order $16$ for which \ref{R2}, and hence also \ref{R1}, has a negative answer. 
\end{enumerate}
\end{maintheorem}

The proof of Theorem~\ref{theo:SmallOrders} is by studying several examples for many of which we obtain explicit descriptions of the group $V/V'$. In all cases where we calculate the isomorphism type of $V/V'$, we find that \ref{E1} does hold.
 These examples might also lead to refined questions on the abelianization of the unit group. Note that traditionally major problems in the field, such as the Zassenhaus Conjectures or questions on finite index subgroups, came from observations on small examples. We are not aware of any negative result to \ref{E1}, \ref{E2} or \ref{P}. A motivation for \ref{P} is also provided by a classical theorem from G. Higman's thesis \cite{San81} that the torsion units of $\mathbb{Z}G$ perfectly reflect the prime spectrum of $G$, i.e., for each prime $p$, $\V(\mathbb ZG)$ contains an element of order $p$, if and only if $G$ does. \\

	For the precise structure of $V/V'$, one naturally looks for a description of $V$, which is one of the fundamental questions in the study of units in group rings. 
	However, the complete answer is known for very few groups. For the dicyclic group of order 12, and for two non-abelian groups of order 16, we provide an explicit presentation, in terms of generators and relations, for unit groups of their integral group rings, which had not been known before. This can be seen as a contribution to \cite[Problem 17]{Sehgal1993}.\\
		 
Another question which found major attention in the study of unit groups of integral group rings, is the existence and  structure of normal complements of the group base $G$ in $V$, see e.g.\ \cite[Section 12.5]{JdR1} or \cite[Chapter 4]{Sehgal1993}. We make use of results in this direction and our calculations to prove: 

\begin{mainproposition}\label{theo:FreeComplement}
Let $G$ be a finite group such that $V = \V(\ZZ G)$ has a free normal complement, i.e.,\ $V = F \rtimes G$ for some infinite cyclic or non-abelian free group $F$. Then $\rank V/V' = \rank \Z(V) = 0$ and $\exp V/V'  = \exp G/G'$, i.e., \ref{R1} and \ref{E1} have positive answers in this case.
\end{mainproposition}

The paper is structured as follows. We begin by discussing some groups and unit groups of orders which have and  some which do not have properties \eqref{var1} and \eqref{var2} in Section~\ref{sec:prelim}. In Section~\ref{sec:BicAndBass}, we study the behavior of Bass and bicyclic units under the abelianization of $\V(\ZZ G)$ and prove Theorem~\ref{theo:ranks} and Proposition \ref{prop:MainBicBass}. We then consider \ref{E1} and \ref{R1} for dihedral groups in Section \ref{sec:Dihedral} and prove Theorem~\ref{prop:Dihedral}. This is followed, in Section~\ref{sec:ExplicitAb}, by the explicit computations of the unit groups and their abelianizations, for integral group rings of groups of small orders, for many of which we obtain \ref{R1} and \ref{E1}. Also, it is shown that \ref{R2} does not hold in general. This section includes the proofs of Theorem~\ref{theo:SmallOrders} and  Proposition \ref{theo:FreeComplement}. We finish with some remarks in Section~\ref{sec:rems}.

\section{Preliminaries and examples}\label{sec:prelim}
Before turning our attention to unit groups of integral group rings, we shortly discuss more general groups and give positive as well as negative examples for properties \eqref{var1} and \eqref{var2}. 

	\begin{example}\label{ex:negativeexamplesto1}
		A counterexample to \eqref{var1} is given by the following: Let $N$ be the direct product of countably many Pr\"ufer $2$-groups $C_{2^\infty}$ and let $x$ be an involution acting on each of these direct factors by inversion. Set $G = N \rtimes \langle x \rangle$. Then $G$ has infinite center, consisting of $1$ and all the involutions in $N$, but has finite abelianization, as $G' = N$.
	\end{example}
	
	\begin{example}\label{ex:negativeexamplesto2}
		A counterexample to \eqref{var2} is already given by a non-abelian free group $F$ as its center is trivial, whereas the rank of its abelianization is the number of generators of $F$. A periodic group with this property is given by the natural wreath product $G = C_p \wr C_{p^\infty}$, where the center of $G$ is trivial while the abelianization is isomorphic to $C_{p^\infty}$. Also, as mentioned in the introduction, \eqref{var2} fails for $\SL_2(\ZZ[\sqrt{-2}])$. Moreover, by \cite[Table~1]{Scheutzow}, also full unit groups of matrix rings over rings of algebraic integers might fail to satisfy \eqref{var2}, for instance $\GL_2(\mathbb{Z}[\sqrt{-21}])$ has infinite abelianization, but finite center. A further example is the unit group of an integral group ring given in Example \ref{ex:Ord16}(ii). More examples  can be found in \cite[Examples 6.4, 6.10]{Wei77}. 
	\end{example}
	
Though \eqref{var1} and \eqref{var2} may not be true in general, it is still interesting to know the classes where they are.

\begin{example}
		As already stated in the introduction, \eqref{var1}  holds for unit groups of integral group rings of finite groups. Moreover, for any commutative ring $R$, \eqref{var1} holds for the groups $\GL_n(R)$, as we have an epimorphism $\det \colon \GL_n(R) \to \U(R)$ onto an abelian group and $\Z(\GL_n(R)) \cong \U(R)$.
\end{example}

\begin{example}
		Groups with an infinite simple derived subgroup satisfy \eqref{var2}. For, if $G$ is a group such that $G'$ is simple and is of finite index in $\V$, then clearly $\Z(G)$ is finite, as otherwise it would have infinite intersection with $G'$. Infinite simple groups have been at the center of attention recently \cite{Hyd17, Nek18, SWZ19, BZ20}. It has been shown by J. Hyde that if $G$ is a so-called vigorous group of homeomorphisms of a Cantor space, the condition of $G'$ being simple can be described in geometric terms (see e.g. \cite[Lemmas 3.3.5, 3.3.8, 3.4.2]{Hyd17} or \cite[Proof of Theorem 3.4]{BZ20}). Another criterion for groups of homeomorphisms of certain topological spaces to have simple commutator subgroup is described in \cite{Eps70}. This contributes an interesting class of examples where \eqref{var2} has a positive answer.
\end{example}
	
	We next discuss the situation for Wedderburn components of the rational group algebras. For many questions on units of integral groups rings certain simple components in $\QQ G$, so-called exceptional components, need special treatment, since no generic results are available. 
	
\begin{definition}\cite[Definition~11.2.2]{JdR1}\label{exc_comp} Let $D$ be a finite dimensional division algebra over $\mathbb{Q}$. The $n \times n$ matrix algebra $\operatorname{M}_n(D)$ is called an \emph{exceptional component}, if it is of one of the following types:
\begin{enumerate}
\item[(I)] $n = 1$ and $D$ is a non-commutative division algebra other than a totally definite quaternion algebra over a number field,
\item[(II)] $n = 2$ and $D$ is $\mathbb{Q}$, an imaginary quadratic extension of $\mathbb{Q}$ or a totally definite quaternion algebra with center $\mathbb{Q}$.
\end{enumerate} 
\end{definition}

The division algebras $D$ in (II) are exactly those division algebras over $\mathbb{Q}$ that have a maximal $\mathbb{Z}$-order with finite unit group \cite[Theorem~2.10]{BJJKT1}.  In general,
(\ref{var2}) fails for exceptional components of type (II), see $\GL_2(\mathbb{Z}[\sqrt{-21}])$ in Example~\ref{ex:negativeexamplesto2}. But for those exceptional components of type (II) that actually appear in the Wedderburn decomposition of rational group algebras $\mathbb{Q}G$ for finite groups $G$, it has a positive answer:

\begin{proposition}\label{prop:WedderbrunComp} Let $G$ be a finite group and let $A$ be a simple component of the rational group algebra $\mathbb{Q}G$, which is not exceptional of type (I) and let $\mathcal{O}$ be a maximal order of $A$. Set $ U = \U(\mathcal{O})$. Then $\rank U/U' = \rank \Z(U)$. In particular, \eqref{var1} and \eqref{var2} have positive answers for $U$. 
\end{proposition}

\begin{proof} If $A$ is not an exceptional component, then $\rank U/U' = \rank \Z(U)$ by \cite[Proof of Theorem 6.3]{JOdRVG}. So assume that $A$ is exceptional. If $A =\operatorname{M}_2(D)$ is of type (II) and actually appears in some rational group algebra $\mathbb{Q}G$, then the division algebra $D$ is one of the fields $\mathbb{Q}, \mathbb{Q}(i), \mathbb{Q}(\sqrt{-2}), \mathbb{Q}(\sqrt{-3})$ or one of the quaternion algebras $\left(\frac{-1,-1}{\mathbb{Q}}\right), \left(\frac{-1,-3}{\mathbb{Q}}\right), \left(\frac{-2,-5}{\mathbb{Q}}\right)$ by \cite[Theorems~3.1 \& 3.5]{EKVG}. The groups $\GL_2(\mathbb{Z})$, $\GL_2(\mathbb{Z}[i])$, $\GL_2(\mathbb{Z}[\sqrt{-2}])$ and $\GL_2(\mathbb{Z}[\frac{-1+\sqrt{-3}}{2}])$ have finite abelianization, see e.g.\ \cite[Corollaries 5.3, 10.3, 6.3]{Swan} for the imaginary quadratic fields. For the non-commutative cases see, e.g.\ \cite[Corollary~3.14]{BJJKT1}.  
\end{proof}

It follows that (\ref{var2}) has a positive answer for maximal orders in $\QQ G$, as for \cut groups, components of type (I) do not appear \cite[Proposition~6.10]{BJJKT1}. Yet for the order $\ZZ G$, the question is more subtle, since it does not reduce to maximal orders in the components, as we will see in Corollary~\ref{cor:D16+}.  On the other hand, as for many questions on division algebras, a general answer whether \eqref{var2} holds for orders in components of type (I) is not known.

\section{Bicyclic units and Bass units}\label{sec:BicAndBass}
Let $G$ be a finite group and let $V = \mathrm{V}(\mathbb{Z}G)$. We want to study the behavior of bicyclic and Bass units under the projection $\varphi: V \rightarrow V/V'$. We show that bicyclic units map to units of finite order and Bass units contribute a free abelian subgroup whose rank is the same as the rank of the center of $V$. Moreover, we show that given a bicyclic unit $u$, the order of $\varphi(u)$ divides $\exp G$ and this also holds for those Bass units which project to elements of finite order. This will also yield the proofs of Theorem~\ref{theo:ranks} and Proposition~\ref{prop:MainBicBass}. For elements $x,y$ in a group $X$, we denote $x$ conjugated by $y$ as $x^y:=y^{-1}xy$ and $[x,y]: = x^{-1}x^y$. Further, if $x,y \in X$ are conjugate in $X$, we write $x \sim_X y$. Note that $u\sim_V v$ clearly implies $\varphi(u) = \varphi(v)$.

\subsection{Bicyclic units}
We will use the following notation, which is a slight modification of that used in \cite{JdR1}. For a subgroup $H$ of $G$ and an element $g$ in $G$, we write $\widetilde{H} = \sum_{h \in H} h \in \mathbb{Z}G$ and $\widetilde{g} = \widetilde{\langle g \rangle}$. Moreover for $g,h \in G$ we denote 
\[b(g,h	): = 1 + (1-h)g\widetilde{h}, \]
a \emph{bicyclic unit} in $\mathrm{V}(\mathbb{Z}G)$. As $\tilde{h}(1-h) = 0$, it is easy to see that $b(g,h)^k = 1 + k(1-h)g\widetilde{h}$ for $k \in \mathbb{Z}$, in particular $b(g,h)^{-1} = 1 - (1-h)g\widetilde{h}$. A  bicyclic unit $b(g,h)$ is trivial, if and only if $g$ normalizes $\langle h \rangle$, and is of infinite order otherwise.\\

We show that any bicyclic unit $u$ of $V$ projects to an element of finite order in $V/V'$. In fact, the order of $\varphi(u)$ divides the order of an element in $G$.

\begin{proposition}\label{prop:bicyclic}
		Let $g, h \in G$ be such that $h$ is of order $n$. Then 
		\[\prod_{k = 1}^n [b(g,h)^{-1}, h^k] = b(g,h)^{n}.\]
		In particular, $\varphi(b(g,h))^n = 1$.
	\end{proposition}
	
	\begin{proof} 
		Observe that $\tilde{h}h^k = \tilde{h}$ for any $k \in \mathbb{Z}$ and that for any element $w \in \mathbb{Z}\langle h \rangle$ of augmentation $0$, we have $\widetilde{h}w = 0$. This gives
		
		\begin{align*}
		[b(g,h)^{-1}, h^k] =& (1+(1-h)g\widetilde{h})h^{-k}(1-(1-h)g\widetilde{h})h^k \\
		=& (1+(1-h)g\widetilde{h})(1-(h^{-k}-h^{-k+1})g\widetilde{h})\\
		=& 1 + (1 - h - h^{-k} + h^{-k+1})g\widetilde{h}.
		\end{align*}
		
		Hence, we obtain
		
		\begin{align*}
		\prod_{k = 1}^n [b(g,h)^{-1}, h^k] &= \prod_{k = 1}^n [1+(1-h- h^{-k} + h^{-k+1})g\widetilde{h}] \\ 
		&= 1 + \sum_{k=1}^n (1-h- h^{-k} + h^{-k+1})g\widetilde{h} \\
		&= 1 + n(1-h)g\widetilde{h} = b(g,h)^{n}. \qedhere
		\end{align*} 
		
\end{proof}

\subsection{Bass units}
If $g \in G$ is of order $n$ and $k$, $m$ are positive integers such that $k$ is coprime to $n$ and $k^m \equiv 1 \mod n$, then
\[u_{k,m}(g):= (1+g+g^2+...+g^{k-1})^m + \frac{1-k^m}{n}\widetilde{g} \]
is a \emph{Bass unit}. A Bass unit $u_{k,m}(g)$, based on a non-trivial element $g$ of $G$ is trivial, i.e., an element of $G$, if and only if $k \equiv \pm 1 \mod n$, and is a unit of infinite order otherwise. The following rules are recorded in \cite{JdR1}, but more visible in \cite{JOdRVG}. If $k, l, m, m', i$ are positive integers such that $k$ and $l$ are coprime to $n$ and $k^m \equiv l^m \equiv k^{m'} \bmod n$, then

 \begin{align}
 u_{k,m}(g) &= u_{l,m}(g), \qquad \text{if} \ k \equiv l \bmod n; \\
 u_{k,m}(g)u_{k,m'}(g) &= u_{k,m+m'}(g), \\
 u_{k,m}(g)u_{l,m}(g^k) &= u_{kl,m}(g), \label{GetBassDown} \\
 u_{1,m}(g) &= 1, \\
 u_{-1,m}(g) &= (-g)^{-m}, \\
 u_{k,m}(g)^i &= u_{k,im}(g), \label{PowersBass} \\
 u_{k,m}(g)^{-1} &= u_{k^{-1},m}(g^k), \qquad \text{where} \ k \cdot k^{-1} \equiv 1 \bmod n, \label{InverseBass} \\ 
 u_{n-k,m}(g) &= u_{k,m}(g)g^{-km}, \qquad \text{if} \ (-1)^m \equiv 1 \bmod n \label{PutMinusBass}.   
 \end{align}

Our goal here is to show that the group generated by the projections of the Bass units and trivial units, i.e., the elements of $G$, in $V/V'$ has the same rank as the center of $V$. 
Moreover, if  for a Bass unit $u$, the element $\varphi(u)$ has finite order, say $n$, then $n$ is the order of an element in $G$.

\begin{lemma}\label{lem_BassTorsion}
Let $g \in G$ be an element of order $n$ and let $l$, $m$ be integers such that $l^m \equiv 1 \mod n$. Assume that $g \sim_G g^l$, say $g^h = g^l$ for some $h \in G$, and let $s$ be the order of $l$ in $\U(\mathbb{Z}/n\mathbb{Z})$. Then 
\[\prod_{i=1}^{s-1} [u_{l,m}(g)^{-1}, h^i] = u_{l,m}(g)^s.\]
In particular, $\varphi(u_{l,m}(g))^s = 1$. 
\end{lemma}

\begin{proof} We write $l^{-1}$ for an integer such that $l\cdot l^{-1} \equiv 1 \mod n$. First of all note that $u_{l,m}(g)^{h^i} = u_{l,m}(g^{l^i})$. So using \eqref{InverseBass} in the last equality we obtain 

\vspace{-.2cm}
\begin{align}
\prod_{i=1}^{s-1} [u_{l,m}(g)^{-1}, h^i]& = 
	\prod_{i=1}^{s-1}u_{l,m}(g)(u_{l,m}(g)^{-1})^{h^i}=u_{l,m}(g)^{s-1}\prod_{i=1}^{s-1}(u_{l,m}(g)^{-1})^{h^i} \nonumber \\
&= u_{l,m}(g)^{s-1} \prod_{i=1}^{s-1} u_{l,m}(g^{l^i})^{-1} = u_{l,m}(g)^{s-1} \prod_{i=1}^{s-1} u_{l^{-1},m}(g^{l^{i+1}}). \label{BassProof1}
\end{align}
Note that $g^{l^{m-i}} = g^{l^{-i}}$. So first reordering factors and then using \eqref{GetBassDown} over and over again, we have

\vspace{-.2cm}
\begin{align*}
\prod_{i=1}^{s-1} u_{l^{-1},m}(g^{l^{i+1}})& = u_{l^{-1},m}(g) u_{l^{-1},m}(g^{l^{s-1}}) u_{l^{-1},m}(g^{l^{s-2}})...u_{l^{-1},m}(g^{l^2}) \\
&=  u_{l^{-1},m}(g) u_{l^{-1},m}(g^{l^{-1}}) u_{l^{-1},m}(g^{l^{-2}})...u_{l^{-1},m}(g^{l^{-(s-2)}}) \\
&=  u_{l^{-2},m}(g) u_{l^{-1},m}(g^{l^{-2}})...u_{l^{-1},m}(g^{l^{-(s-2)}}) = u_{l^{-3},m}(g) u_{l^{-1},m}(g^{l^{-3}})...u_{l^{-1},m}(g^{l^{-(s-2)}}) \\
&= ... = u_{l^{-(s-2)},m}(g) u_{l^{-1},m}(g^{l^{-(s-2)}}) = u_{l^{-(s-1)},m}(g) =u_{l,m}(g).
\end{align*}
Plugging this equality into \eqref{BassProof1}, we obtain the claim of the lemma.
\end{proof}

The following fact follows directly from the formulas on Bass units, but is quite useful for us.
\begin{lemma}\label{lem_lkBass}
Let $g \in G$ be of order $n$ and let $l$ be an integer such that $g \sim_G g^l$. Assume that $k$ is an integer coprime to $n$ and $m$ is a positive integer satisfying $k^m \equiv l^m \equiv 1 \mod n$. Then
\[\varphi(u_{k,m}(g)u_{l,m}(g)) = \varphi(u_{kl,m}(g)). \]
\end{lemma}

\begin{proof}
Let $h \in G$ such that $g^h = g^l$. Then $u_{k,m}(g)^h = u_{k,m}(g^l)$, so $\varphi(u_{k,m}(g)) = \varphi(u_{k,m}(g^l))$. Hence using \eqref{GetBassDown} we have
\[\varphi(u_{k,m}(g)u_{l,m}(g)) = \varphi(u_{k,m}(g^l)u_{l,m}(g)) = \varphi(u_{kl,m}(g)).  \qedhere\]
\end{proof}

We are now ready to prove our main observation of this section. 

\begin{proposition}\label{prop:Bass}
Let $\mathcal{B}_{1}$ be the subgroup of $V$ generated by the Bass units and trivial units in $V$. Then $\rank \varphi(\mathcal{B}_{1}) = \rank \Z(V)$. Moreover, $\exp \varphi(\mathcal{B}_1)$ divides $\exp G$.
\end{proposition}

\begin{proof}
We first show that $\rank \Z(V) \leqslant \rank \varphi(\mathcal{B}_1)$. Denote by $\GL(\ZZ G)$ the group of automorphisms of the free $\ZZ G$-module of countable rank, i.e., the group of ``infinite matrices'' over $\ZZ G$ which are invertible. The group $V$ embeds in $\GL(\ZZ G)$ by acting on a rank $1$ submodule. Recall that $K_1(\ZZ G) = \GL(\ZZ G)/\GL(\ZZ G)'$ and note $V' \leqslant \GL(\ZZ G)'$. We hence have a (non-exact) sequence of surjections
\[\mathcal{B}_1  \twoheadrightarrow \varphi(\mathcal{B}_1) (= \mathcal{B}_1/(\mathcal{B}_1 \cap V') ) 
\twoheadrightarrow \mathcal{B}_1/(\mathcal{B}_1 \cap \GL(\ZZ G)'). \]
The last term is a subgroup of $K_1(\ZZ G)$ which, by \cite[Corollary 11.1.3]{JdR1}, has finite index in $K_1(\mathbb{Z}G)$.
Hence $\rank K_1(\mathbb{Z}G) \leqslant \rank \varphi(\mathcal{B}_1)$. Also, by \cite[Corollary 9.5.10]{JdR1}, we have $\rank K_1(\mathbb{Z}G) = \rank \Z(V)$.  Therefore, $\rank \Z(V) \leqslant \rank \varphi(\mathcal{B}_1)$.

We next show that $\rank \varphi(\mathcal{B}_1) \leqslant \rank \Z(V)$. Note that the rank of the center of $V$ equals $n_\mathbb{R} - n_\mathbb{Q}$, where $n_\mathbb{R}$ denotes the number of $\mathbb{R}$-classes in $G$ and $n_\mathbb{Q}$ denotes the number of $\mathbb{Q}$-classes in $G$.
Here an $\mathbb{R}$-class is a set of the form $g^G \cup (g^{-1})^G$ for $g \in G$ and a $\mathbb{Q}$-class of an element $g$ is defined as $\cup_{(i, |g|) = 1} (g^i)^G$ (so $n_\mathbb{Q}$ equals the number of conjugacy classes of cyclic subgroups in $G$).
A proof for this well-known fact can be found e.g. in \cite[Proposition 2.1]{BBM20}.

Note that if for some $g, g' \in G$ we have $g\sim_G g'$, then 
$u_{k,m}(g)\sim_V u_{k,m}(g')$ and so $\varphi(u_{k,m}(g)) = \varphi(u_{k,m}(g'))$, for any suitable integers $k,m$. Also, in view of \eqref{GetBassDown} the group generated by the elements $u_{k,m}(g)$, where $k$ and $m$ vary over all admissible integers, is the same as the group generated by the elements $u_{k,m}(g^i)$, for any integer $i$ coprime to $n$. In particular, the group generated by the image of Bass units
 $u_{k,m}(g)$ in $V/V'$, where $k$ and $m$ vary over all admissible integers, is the same for any choice of $g$ in a $\mathbb{Q}$-conjugacy class. 

  Let $g\in G$ and let $W$ be the subgroup of $V/V'$ generated by all elements of the form $\varphi(u_{k,m}(g))$, where $k$ and $m$ take all admissible values. It hence suffices to show that the rank of $W$ is at most the number of $\mathbb{R}$-classes contained in the $\mathbb{Q}$-class of $g$ subtracted by $1$, as our goal is to obtain at most the number $n_\mathbb{R} - n_\mathbb{Q}$ when varying over all $g \in G$.
  
  Let $m_n$ be the number of positive integers  smaller than $n$ which are coprime to $n$, i.e., the value of Euler-totient function on $n$. Hence, $k^{m_n} \equiv 1 \mod n$, for any integer $k$ coprime to $n$. Moreover, in view of \eqref{PowersBass}, for any $m$ satisfying $k^{m} \equiv 1 \mod n$, we observe that $u_{k,m}(g)$ has finite order modulo $\langle u_{k,m_n}(g) \rangle$, as $m_n$ and $m$ are both multiples of  the multiplicative order of $k$ modulo $n$. It hence suffices to consider only the elements of the form $u_{k,m_n}(g)$.
  
  Let $L = \{1 \leqslant l \leqslant n \ | \ g \sim_G g^l \}$. Then the number of real classes in the rational class of $g$ equals $\frac{m_n}{|L|}$, if $-1 \in L$, and $\frac{m_n}{2|L|}$ otherwise. First note that, in view of Lemma~\ref{lem_BassTorsion}, if $l \in L$, then $\varphi(u_{l,m_n}(g))$ has finite order. Now, as $u_{n-k,m}(g) = u_{k,m}(g)g^{-km}$ (by \eqref{PutMinusBass}), we have that $\varphi(u_{k,m_n}(g))$ and $\varphi(u_{-k,m_n}(g))$ differ only by an element of finite order. Moreover, by Lemma~\ref{lem_lkBass}, this is also the case for $\varphi(u_{k,m_n}(g))$ and $\varphi(u_{lk,m_n}(g))$ for any $l \in L$, as $\varphi(u_{l,m_n}(g))$ has finite order. Therefore, the number of elements of the form $\varphi(u_{k,m_n}(g))$ which generate infinite cyclic groups, which could be different modulo elements of finite order is at most $(\frac{m_n}{|L|}-1)$, if $-1 \in L$ and $(\frac{m_n}{2|L|}-1)$, otherwise. This means exactly that the rank of $W$ is at most the number of $\mathbb{R}$-classes contained in the $\mathbb{Q}$-class of $g$ subtracted by 1. We thus obtain that $\rank \varphi(\mathcal{B}_1) \leqslant n_\mathbb{R} - n_\mathbb{Q} = \rank \Z(V)$. As we obtained $\rank \Z(V) \leqslant \rank \varphi(\mathcal{B}_1)$ already before, we have $\rank \Z(V) = \rank \varphi(\mathcal{B}_1)$. 
  
This equality of ranks together with the arguments in the last paragraph shows that $\varphi(u_{k,m}(g))$, for suitable $k$ and $m$, has finite order, if and only if $g \sim_G g^{\pm k}$. Moreover, by Lemma~\ref{lem_BassTorsion}, if $g$ is of order $n$ and $g \sim_G g^k$, then the order of $\varphi(u_{k,m}(g))$ divides $s$, the order of $k$ in $\U(\mathbb{Z}/n\mathbb{Z})$. As there is some $h \in G$ such that $g^h = g^k$, we conclude that the order of $h$ is divisible by $s$ and so the order of $\varphi(u_{k,m}(g))$ divides $\exp G$. Now consider $u_{n-k,m}(g)$ under the assumption $g \sim_G g^k$. As $u_{n-k,m}(g) = u_{k,m}(g)g^{-km}$ by \eqref{PutMinusBass} we have $\varphi(u_{n-k,m}(g)) = \varphi(u_{k,m}(g))\varphi(g^{-km})$. By the observation before, the order of the last expression divides $\exp G$. Hence, we obtain that $\exp \varphi(\mathcal{B}_1)$ divides $\exp G$.  
\end{proof}

\begin{proof}[Proof of Proposition~\ref{prop:MainBicBass}]
This is an immediate consequence of Propositions~\ref{prop:bicyclic} and \ref{prop:Bass}. 
\end{proof}

\subsection{Finite index subgroups}
The next easy general group-theoretical observation allows us to prove Theorem \ref{theo:ranks}.

\begin{lemma}\label{lem:FiniteIndexGeneral}
Let $V$ be a group and let $B$ be a subgroup of finite index in $V$. Then $\rank V/V' \leqslant \rank B/B'$. 
\end{lemma}
\begin{proof}
Let $\varphi: V \rightarrow V/V'$ be the natural projection. Then $\varphi(B)$ has finite index in $\varphi(V)$, as $v^n \in B$ for some $v \in V$ implies $\varphi(v)^n \in \varphi(B)$. So, $\rank \varphi(B) = \rank \varphi(V)$. Now consider $\tau: {B/B' \rightarrow \varphi(B),}$ ${bB' \mapsto bV'}$. As $B'$ is a subgroup of $V'$, we have that $\tau$ is surjective and thus $\rank B/B' \geqslant \rank \operatorname{im}(\tau) = \rank \varphi(B).$  Putting together, we have $\rank V/V' \leqslant \rank B/B'$. 
\end{proof}

We are now ready to prove that if the known generic constructions of units generate a large subgroup in $\U(\ZZ G)$, then the questions on ranks have a affirmative answers in its strong form.\\

\noindent\textbf{Theorem~\ref{theo:ranks}.}\ \emph{  Let $G$ be a finite group and let $\mathcal{B}$ be the subgroup of $V = \V(\ZZ G)$, generated by the elements of $G$, the bicyclic and the Bass units of $\ZZ G$. If $\mathcal{B}$ has finite index in $V$, then $\rank V/V' = \rank \Z(V)$, i.e., \ref{R1} has a positive answer.}

\begin{proof} Let $\mathcal{B}$ be the subgroup of $V = \mathrm{V}(\mathbb{Z}G)$ generated by the trivial, bicyclic and Bass units. Then, by Propositions~\ref{prop:bicyclic} and \ref{prop:Bass}, we know that $\rank \mathcal{B}/\mathcal{B}' \leqslant \rank \Z(V)$. As $\mathcal{B}$ has finite index in $V$ by assumption, from \cite[Proposition~6.1]{BJJKT1} and Lemma~\ref{lem:FiniteIndexGeneral}, we obtain
\[\rank\Z(V) \leqslant \rank V/V' \leqslant \rank \mathcal{B}/\mathcal{B}' \leqslant \rank \Z(V), \]
which implies the theorem. \end{proof}

\section{Dihedral groups}\label{sec:Dihedral}

In this section, we study the abelianization of $\V(\ZZ G)$ for a dihedral group $G$. We show that (R1) holds for all dihedral groups and that also (E1) holds for dihedral groups of order $2p$. 

\begin{corollary}\label{cor:dihedral}
Let $G$ be a finite dihedral group and let $V = \V(\ZZ G)$. Then $\rank \Z(V) = \rank V/V'$, i.e.,\ \ref{R1} has a positive answer.
\end{corollary}
\begin{proof}
By \cite[Theorem 23.1]{Sehgal1993}, the group generated by bicyclic and Bass units in $V$ has finite index in $V$. So, \ref{R1} has a positive answer for $G$ by Theorem~\ref{theo:ranks}. 
\end{proof}

For the proof of (E1) for dihedral groups of order $2p$, we use the description of the unit group given by Passman and Smith \cite{PS81} and a result on congruence subgroups.\\

\noindent\textbf{Theorem~\ref{prop:Dihedral}.}\ \emph{ Let $G$ be a dihedral group of order $2p$, where $p$ is an odd prime and let $V = \V(\ZZ G)$.  Then, $\exp G/G'= \exp V/V'$, i.e., (E1) holds for $G$.}

	{\begin{proof} Let $G$ and $V$ be as in the statement of 				the theorem. Since $\U(\ZZ G) \cong V \times C_2$ and $G/G'$ has order $2$, we show that the abelianization of $\U(\ZZ G)$ is an elementary abelian $2$-group, to obtain the desired result. For this, we use the description of $\U(\ZZ G)$ from \cite[Theorem 3.7]{PS81}. 
				
		Let $\zeta$ be a primitive complex $p$-th root of unity, $z = \zeta+\zeta^{-1}$ and $R = \ZZ[z]$. Let $Q$ be the maximal ideal of $R$ containing $p$ or equivalently, in the language of \cite{PS81}, the intersection of the ideal $(\zeta-\zeta^{-1})\ZZ[\zeta]$ of $\ZZ[\zeta]$ with $R$. Note that $R/Q \cong \mathbb{F}_p$ and  $z \equiv 2 \mod Q$.  By \cite[Theorem 3.7]{PS81}, the group $\U(\ZZ G)$ is isomorphic to the matrix group
		\[U := \left\{ A = \begin{pmatrix} a & b \\ c & d \end{pmatrix} \in \operatorname{GL}_2(R) \ \middle| \ \det(A) \equiv \pm 1 \bmod Q \ \ \text{and} \ \ a+c \equiv b + d \equiv \pm 1 \bmod Q \right\} \]
		and we identify $\U(\ZZ G)$ with this group from now on. 
		
		Let \[U_1 = U \cap \SL_2(R) \qquad \text{and}\qquad  D = \left\langle\ \begin{pmatrix} \lambda & 0 \\ 0 & 1 \end{pmatrix} \ \middle| \ \lambda \in \U(R), \ \lambda \equiv 1 \bmod Q \ \right\rangle \leqslant U.\]\\
		We first claim that $U$ is generated by $U_1$, $D$ and $\left(\begin{smallmatrix} 0 & 1 \\ 1 & 0 \end{smallmatrix}\right)$. 
		Note that the last matrix is in $U$. To prove the claim, let $A =\left(\begin{smallmatrix} a & b \\ c & d \end{smallmatrix}\right) \in U$ and $\det(A) = \delta \equiv \pm 1\bmod Q$. If $\delta\equiv 1\bmod Q$, then also $\delta^{-1} \equiv  1 \bmod Q$, so 
		 that 
		\[\begin{pmatrix}\delta^{-1} & 0 \\ 0 & 1 \end{pmatrix} A  \in U_1.\]
		Hence $A = \left(\begin{smallmatrix} \delta & 0 \\ 0 & 1 \end{smallmatrix}\right) A_1$, where $A_1 \in U_1$ and  $\left(\begin{smallmatrix} \delta & 0 \\ 0 & 1 \end{smallmatrix}\right) \in D$. Moreover, if $\delta \equiv -1 \bmod Q$, then $\left(\begin{smallmatrix} -\delta & 0 \\ 0 & 1 \end{smallmatrix}\right) \in D$ and in this case 
		\[\begin{pmatrix} 0 & 1 \\ 1 & 0 \end{pmatrix} \begin{pmatrix} -\delta^{-1} & 0 \\ 0 & 1 \end{pmatrix}A \in U_1, \]
		 so that  $A = \left(\begin{smallmatrix} -\delta & 0 \\ 0 & 1 \end{smallmatrix}\right) \left(\begin{smallmatrix} 0 & 1 \\ 1 & 0 \end{smallmatrix}\right)A_1$ for some $A_1\in U_1$ and $\left(\begin{smallmatrix} -\delta & 0 \\ 0 & 1 \end{smallmatrix}\right) \in D$.
		
		As $Q$ is of finite index in $R$, for any unit $r \in \U(R)$ there is a number $k$ such that $r^k \equiv 1 \mod Q$. This means that for the abelian group $D$ we have $\rank D = \rank \U(R)$. Moreover, as $R$ is a totally real ring, in view of Dirichlet's Unit Theorem (see e.g.\ \cite[Theorem 5.2.4]{JdR1}), $\rank \U(R)$ equals $m_\mathbb{R} - 1$, where $m_\mathbb{R}$ is the number of real embeddings of $R$. Now, the real embeddings of $R$ are given by $z \mapsto \zeta^k + \zeta^{-k}$, for some $k \in \ZZ$ such that $k \not\equiv 0 \bmod p$. Hence, there are $\frac{p-1}{2}$ distinct real embeddings of $R$ and therefore $ \rank D = \rank \U(R) = \frac{p-3}{2}$. Moreover, $\rank \Z(U) = \rank\Z(\U(\ZZ G)) = \frac{p-3}{2}$, as it equals $n_\mathbb{R}-n_\mathbb{Q}$, the difference in number of real classes and the number of rational classes in $G$. Also,  by Corollary~\ref{cor:dihedral}, $\rank \Z(U) = \rank \varphi(U)$, where $\varphi: U \rightarrow U/U'$ denotes the natural projection. Hence, $\rank D = \rank\varphi(U)= \frac{p-3}{2}$. 	
		
		We first observe that it suffices to show  that $\varphi(U_1)$ is an elementary abelian $2$-group.  For this, note that because $D$, $U_1$ and $\left(\begin{smallmatrix} 0 & 1 \\ 1 & 0 \end{smallmatrix}\right)$ generate $U$, we have that $\frac{p-3}{2} = \rank \varphi(U) \leqslant \rank \varphi(U_1) + \rank\varphi(D)$, as of course the order of $\varphi\left(\left(\begin{smallmatrix} 0 & 1 \\ 1 & 0 \end{smallmatrix}\right)\right)$ divides $2$. Since $\rank \varphi(D)\leqslant \rank D = \frac{p-3}{2}$, this implies $\rank \varphi(D) = \frac{p-3}{2}$, which in other words means that no element of infinite order in $D$ is mapped to an element of finite order by $\varphi$ and so $\exp \varphi(D) \mid \exp D$. Furthermore, the torsion subgroup of $D$ is isomorphic to a torsion subgroup of $\U(R)$ which, again by Dirichlet's Theorem, is isomorphic to $C_2$, as $-1$ is the only non-trivial root of unity in $R$. Hence, if 	$\varphi(U_1)$ is an elementary abelian $2$-group, then so is $\varphi(U)$.
		
		Let  $\Gamma_Q$ be the congruence subgroup of $\SL_2(R)$ with respect to the ideal $Q$, i.e., the kernel of the natural map $\SL_2(R) \rightarrow \SL_2(R/Q) \cong \SL_2(\mathbb{F}_p)$. Clearly $\Gamma_Q \leqslant U$. We first show that $\varphi(\Gamma_Q)$ is an elementary abelian $2$-group. By \cite[Section 2.6, Corollary 3]{Ser70}, the group $\Gamma_Q$ is generated by the $\SL_2(R)$-conjugates of the matrices $E(q):=  \left(\begin{smallmatrix} 1 & q \\ 0 & 1 \end{smallmatrix}\right)$, for $q \in Q$. We show that for each $\SL_2(R)$-conjugate $X$ of a matrix $E(q)$, for $q \in Q$, there is an element $Y \in U$ such that
		$X^{Y}=X^{-1}$ and hence $\varphi(\Gamma_Q)$ is an elementary abelian $2$-group. 
		
		For this, first set $Y = \left(\begin{smallmatrix} -1 & -2 \\ 0 & 1 \end{smallmatrix}\right) \in U$. Then $E(q)^{Y } = E(-q) = E(q)^{-1}$. Now let $S = \left(\begin{smallmatrix} a & b \\ c & d \end{smallmatrix}\right) \in \SL_2(R)$ and $E = E(q)^{S}$. Set $C(t):= \left(\begin{smallmatrix} 1 & t \\ 0 & 1 \end{smallmatrix}\right)$ for $t \in R$ and note that $C(t)\in \textup{C}_{\SL_2(R)}(E(q))$, the centralizer of $E(q)$ in $\SL_2(R)$, so that $E = E(q)^{C(t) S}$. It will hence suffice to show that there exists a $t \in R$ such that ${Y }^{C(t)S}$ is an element in $U$, as the property $E(q)^{Y } = E(q)^{-1}$ is of course invariant under conjugation.
		
		Using $ad -bc = 1$, i.e., the fact that $S$ lies in $\SL_2(R)$, we obtain 
		\[{Y }^{S} = \begin{pmatrix} -1-2bc -2cd & -2bd-2d^2 \\ 2ac+2c^2 & -1+2ad+2cd \end{pmatrix}. \]
		Moreover, $C(t)S= \begin{pmatrix} a+tc & b+td \\ c & d \end{pmatrix}$, so
		\[{Y }^{C(t)S} = \begin{pmatrix} -1-2(b+td)c -2cd & -2(b+td)d-2d^2 \\ 2(a+tc)c+2c^2 & -1+2(a+tc)d+2cd \end{pmatrix}. \]
		We want to find a $t \in R$ such that this matrix lies in $U$. As the determinant of this matrix is just the determinant of $Y $, i.e.\,$-1$, this means that the sums of the columns should be congruent to each other and congruent to $\pm 1$ modulo $Q$. Now the sum of the first column is $-1+2c(-b-td-d+a+tc+c)$ and the sum of the second column is $-1+2d(-b-td-d+a+tc+c)$. Hence we have to find $t$ such that

\vspace{-.2cm}		
		\begin{align}\label{eq:colsum}
		c(-b-td-d+a+tc+c) \equiv d(-b-td-d+a+tc+c) \equiv e \bmod Q 
		\end{align}
		
		where $e \in \{0,1\}$. If $c\equiv d \bmod Q $, then using the fact $ad-bc=1,$ we observe that (\ref{eq:colsum}) is satisfied.
	
		In case, $c \not\equiv d \bmod Q$, then 
		\[-b-td-d+a+tc+c \equiv 0 \bmod Q \quad \Leftrightarrow \quad t \equiv \frac{b+d-a-c}{c-d} \bmod Q \]
		and we can pick any $t \in R$ satisfying this condition so that we have $e = 0$ in \eqref{eq:colsum}. 
		
	Therefore,  in order to complete the proof, that also $\varphi(U_1)$ is an elementary abelian $2$-group, it only remains to show that if $A_1\in U_1 \setminus \Gamma_Q$, then the order of $\varphi(A_1)$ is a divisor of $2$. Note that $\Gamma_Q$ is actually a normal subgroup of $U$ as it is a normal subgroup of $\GL_2(R)$. Let $\pi: U \rightarrow U/\Gamma_Q \leqslant \GL_2(\mathbb{F}_p)$ be the natural projection. 
		From the conditions	on the elements of $U_1$, we get that $\pi(U_1)$ consists of matrices of the form $\left(\begin{smallmatrix} 1+a & a \\ -a & 1-a \end{smallmatrix}\right)$ and $\left(\begin{smallmatrix} -1+a & a \\ -a & -1-a \end{smallmatrix}\right)$ for some $a \in \mathbb{F}_p$. These  matrices are $2p$ in number and form a cyclic group of order $2p$, the involution in the group being $\left(\begin{smallmatrix} -1 & 0 \\ 0 & -1 \end{smallmatrix}\right)$ and an element of order $p$ is given by any matrix  $T(a): = \left(\begin{smallmatrix} 1+a & a \\ -a & 1-a \end{smallmatrix}\right)$, for $a \neq 0$. Denote by $A$, a matrix of $U$ such that $\pi(A) = T(a)$ for some $a \neq 0$. Thus any element in $U_1$ can be written as $\left(\begin{smallmatrix} -1 & 0 \\ 0 & -1 \end{smallmatrix}\right)^i A^j \gamma$ for some integers $i,j$ and some $\gamma \in \Gamma_Q$. It will hence be sufficient to show that the order of $\varphi(A)$ is divisible by $2$, as we have already shown this for $\varphi(\gamma)$ and it is obviously true for $\varphi(\left(\begin{smallmatrix} -1 & 0 \\ 0 & -1 \end{smallmatrix}\right))$. 
		
		For this, let $\varphi': U/\Gamma_Q \rightarrow (U/\Gamma_Q)/(U/\Gamma_Q)'$ be the natural projection. Note that $(U/\Gamma_Q)/(U/\Gamma_Q)' = U/(U'\Gamma_Q)$, as it is just the biggest abelian quotient of $U$ which contains $\Gamma_Q$ in its kernel. Let $\pi': U/U' \rightarrow U/(U'\Gamma_Q)$ be the natural projection. Hence, we have a commutative diagram \begin{center}
			\begin{tikzcd}
			U \arrow[r, "\varphi"]\arrow[d, "\pi"] & U/U' \arrow[d, "\pi'"] \\ 
			U/\Gamma_Q \arrow[r, "\varphi'"] & U/(U'\Gamma_Q)
			\end{tikzcd}
		\end{center}
		Now $T(a)^{\pi\left(\left(\begin{smallmatrix} 0 & 1 \\ 1 & 0 \end{smallmatrix}\right)\right)} = T(a)^{-1}$, so that
		 $\varphi'(T(a)^2) = 1$ and hence $\varphi'(T(a)) = 1$, as $T(a)$ has order $p$. 
		Consequently, $\varphi'(\pi(A)) = 1 = \pi'(\varphi(A))$. As $\ker \pi' = \varphi(\Gamma_Q)$ is an elementary abelian $2$-group by our previous calculations, 
		 we get that the order of $\varphi(A)$ 
		 is divisible by $2$. 
	 \end{proof}}
	In the next section, the abelianization of the dihedral groups of order $6$ and $8$ is calculated explicitly, see Examples~\ref{ex:S3} and \ref{ex:D8}.

\section{Explicit abelianization of some unit groups}\label{sec:ExplicitAb}

In this section, we analyze the abelianization of the (normalized) unit group $V:=\V(\mathbb{Z}G)$ of an integral group ring, for a group $G$, of order at most $16$. In the process, we also compute the explicit structure of the full unit group $V =\V(\mathbb{Z}G)$ for certain of these groups, which is otherwise known for very few cases. This can be seen as a contribution to \cite[Problem 17]{Sehgal1993}. We use \textsf{GAP} \cite{GAP4} in some of our proofs and the code to check our claims can be downloaded from the GitHub repository \cite{BMMGit}. 
	Eventually, we prove the following:\\

\noindent\textbf{Theorem~\ref{theo:SmallOrders}.}\ \emph{Let $G$ be a group and let $V = \V(\ZZ G)$.
\begin{enumerate}
\item If $G$ is of order at most $15$, then \ref{R1} and \ref{E1} have positive answers for $G$.
\item There are non-abelian groups of order $16$ for which \ref{R1} has a positive answer. There is a group of order $16$ for which \ref{R2}, and hence also \ref{R1}, has a negative answer. \\
\end{enumerate} } 

If $G$ is an abelian group, then so is $\V(\mathbb{Z}G)$ and hence $V/V'$ is finite, if and only if $G$ is an abelian \cut group, i.e., an abelian group of exponent $1,\ 2,\  3,\ 4$ or $6$. Moreover, if $G=E\times Q_8$, where $E$ denotes elementary abelian 2-group and $Q_8$ denotes the quaternion group of order 8, then $V=G$ and $V/V'=G/G'=E\times C_2\times C_2.$ In all other cases, we have $V\supsetneq G$ \cite{Hig40}. Furthermore, if $G$ is any abelian group (not necessarily of exponent $1,\ 2,\ 3,\ 4$ or $6$), then $V/V'=V=G\times F$, $F$ being a finitely generated group of rank $\frac{1}{2}(|G|+n_2-2c+1)$, where $|G|$ denotes the order of the group $G$, $n_2$ is the number of elements of order $2$ in $G$ and $c$ is the number of cyclic subgroups of $G$ \cite[Theorem 4]{AA69}. Clearly, for all these groups, \ref{R1} and \ref{E1} hold. Henceforth, we only consider non-abelian groups, which have non-trivial units in the integral group ring.

Throughout this section, we denote by $o(v)$, the order of element $v$ and by $\overline{v}$, we denote the image of $v\in V$ under the natural projection $\varphi: V \rightarrow V/V'$. \\

\vspace{-.2cm}
\begin{example}\label{ex:S3}[\textbf{Groups of order 6}] 
	$S_3:=\langle a,b\mid  a^3=1,b^2=1,a^b=a^{-1}\rangle$. By \cite{JP92} (see also \cite{DJ03}, Theorem 2.1), we have $\V(\mathbb{Z}S_3) = \langle u_{0},u_{1},u_{2}\rangle  \rtimes S_{3}$, 
	where $u_{i}:=b(a,a^ib),~ 0\leqslant i \leqslant 2$ (notation as in Section~3). Observe that $u_{0}^a=u_1,~u_{1}^a=u_{2},~ u_2^a=u_0,~u_{0}^b=u_{0}^{-1},~u_1^b=u_2^{-1},~u_2^b=u_1^{-1}$ implying that 
	\begin{small}	\[\V(\mathbb{Z}S_{3})=\langle a,b, u_{0},u_{1},u_{2}\mid a^3=b^2=1, a^b=a^{-1}, u_{0}^a=u_1,~u_{1}^a=u_{2}, u_2^a=u_0,~u_{0}^b=u_{0}^{-1},~u_1^b=u_2^{-1},~u_2^b=u_1^{-1}  \rangle\]	\end{small}
	and hence direct calculations yield,
	\[\V(\mathbb{Z}S_{3})/\V(\mathbb{Z}S_{3})'= \langle  \overline{b },\overline{u_0} \mid \overline{b}^2=\overline{u_0}^2=\overline{1}, [ \overline{b },\overline{u_0}]=\overline{1}\rangle\cong C_2\times C_{2}.\]
	We conclude that \ref{R1} and \ref{E1} hold for $S_3$. Alternatively, one could conclude that \ref{R1} and \ref{E1} hold also using Proposition~\ref{prop:bicyclic} and the description of the generators as bicyclic and trivial units.
\end{example}

\vspace{-.2cm}
\begin{example}\label{ex:D8}[\textbf{Groups of order 8}] The quaternion group of order $8$, as discussed above trivially satisfies \ref{E1} and \ref{R1}. Now, for $D_8:=\langle a,b\mid a^4=1,b^2=1,a^b=a^{-1}\rangle$, using \cite[Theorem 3.1]{DJ03}, we have that

\vspace{-.2cm}
 \begin{footnotesize} \begin{align*} \V(\mathbb{Z}D_{8})=\langle & a,b, u_{0},u_{1},u_{2}\mid  a^4=b^2=1, a^b=a^{-1}, u_{0}^a=u_2,~u_{1}^a=(u_0u_1u_2)^{-1},u_{2}^a=u_0,~u_{0}^b=u_{0}^{-1},~u_1^b=u_0u_1u_2,~u_2^b=u_2^{-1}  \rangle,\end{align*} 	\end{footnotesize}
where $u_{i}:=b(a,a^ib),~ 0\leqslant i \leqslant 2 $ and therefore,
\[\V(\mathbb{Z}D_{8})/\V(\mathbb{Z}D_{8})'= \langle \overline{a}, \overline{b },\overline{u_0},\overline{u_1}\mid \overline{a}^2=\overline{b}^2=\overline{u_0}^2=\overline{u_1}^2=\overline{1}\rangle\cong C_{2}^4.\]
Again, \ref{R1} and \ref{E1} hold for $D_8$, and this last conclusion also follows from Proposition~\ref{prop:bicyclic}.
\end{example}

It may be noted that Examples \ref{ex:S3} and \ref{ex:D8} have also been studied in \cite{SGV97} and \cite{SG01} respectively. We recalculated for convenience and completeness.

\begin{example}\label{ex:Order12}[\textbf{Groups of order 12}] 
We next consider non-abelian groups of order 12 and show that \ref{R1} and \ref{E1} hold for all of these groups. The three non-abelian groups of order 12 are:
\begin{enumerate}
		\item[(i)] $A_4$, the alternating group on 4 elements;	
		\item[(ii)]\label{ex:D12} $D_{12}:=\langle a,b \mid a^6=1,b^2=1,a^b=a^{-1}\rangle$, the dihedral group of order 12; and	
		\item[(iii)] $T:=\langle a,b\mid  a^6=1,b^2=a^3,a^b=a^{-1}\rangle$, the dicyclic group of order $12$ (\texttt{IdSmallGroup(T)\,=\,[12,1]}).\\
		
	\end{enumerate}

	Observe that all these groups are \cut  groups. It turns out that for all these groups $G$, the abelianization of the unit group $\V(\mathbb{Z}G)$ is finite.\\
	
\begin{enumerate}

	\item[(i)] It has been shown in \cite{SG00} that for $V= \V(\ZZ A_4)$ we have $V/V' \cong C_3$, so \ref{R1} and \ref{E1} hold.
	\ref{R1} could also be concluded without explicit calculations, as $\mathbb{Q}A_4$ does not have any exceptional component.\\

\item[(ii)] Next, let $G = D_{12}$. 
	By Corollary~\ref{cor:dihedral} we know that \ref{R1} holds for $G$.
	Moreover, by \cite[Theorem 2]{Jes95}, the group $V$ is generated by bicyclic and trivial units. 

	Recall that a bicyclic unit $b(g,h)$ is non-trivial if and only if $g$ does not normalize $\langle h \rangle$. As elements of order $3$ or $6$ generate normal subgroups in $G$, whenever $b(g,h)$ is a non-trivial bicyclic unit, the element $h$ must be of order $2$. Hence, we conclude by Proposition~\ref{prop:bicyclic} and the fact that $G/G' \cong C_2 \times C_2$ that $V/V'$ is an elementary abelian $2$-group. Therefore, \ref{E1} also holds. \\
	
\item[(iii)] We calculate the abelianization of $\V(\ZZ T)$ explicitly.

	 Let  $G=T:=\langle a,b\mid  a^6=1,b^2=a^3,ba=a^5b\rangle$.  The unit group $V = \mathrm{V}(\mathbb{Z}G)$ was studied in \cite[Theorem 4]{Par93} using the same notation for the generators of the group. We have $\mathrm{V}(\mathbb{Z}G)= C \rtimes G$ where $C$ is a free group of rank 5. In \cite[Theorem 4]{Par93} explicit generators of $C$ are given.
	
	To obtain the relations in the group $V/V'$, we write down the conjugates $x_i^a$ and $x_i^b$ for each $1 \leqslant i \leqslant 5$, obtained using GAP: 

\vspace{10pt}
\begin{minipage}{.33\textwidth}
\begin{itemize}
			\item[(1a)] $x_1^a= x_2$
			\item[(1b)] $x_1^b = x_1^{-1}$
			\item[(2a)] $x_2^a = x_3$
			\item[(2b)] $x_2^b = x_3^{-1}$
\end{itemize}
\end{minipage}
\begin{minipage}{.33\textwidth}
\begin{itemize}
			\item[(3a)] $x_3^a = x_1$
			\item[(3b)] $x_3^b = x_2^{-1}$
			\item[(4a)] $x_4^a = x_2x_4^{-1}x_5$ 
			\item[(4b)] $x_4^b =  x_1^{-1}x_5x_3^{-1}$
\end{itemize}
\end{minipage}
\begin{minipage}{.33\textwidth}
\begin{itemize}
			\item[(5a)] $x_5^a = x_2x_4^{-1}x_1$
			\item[(5b)] $x_5^b = x_1^{-1}x_4x_2^{-1}$
\end{itemize}
\end{minipage}
\vspace{0.3cm}

Clearly, (1a),(2a) and (3a) imply that $\overline{x_{1}}=\overline{x_{2}}=\overline{x_{3}}$ and it follows from (1b) that $o(\overline{x_{1}})$ divides~$2$. 
In view of these conclusions and the fact that $V/V'$ is abelian, (4b) yields that $\overline{x_{4}}=\overline{x_{5}}$. Consequently, (4a) yields that $\overline{x_{4}}=\overline{x_{1}}$. Overall, $\overline{x_i} = \overline{x_j}$ for all $1 \leqslant i,j \leqslant 5$ and we conclude from all relations that $\overline{x_1}$ has order $2$.

 Furthermore, $\overline{a}=\overline{b}^2$ and $\overline{a}=\overline{a}^{-1}$. Putting together, we obtain $V/V'\cong \langle \overline{x_{1}},\overline{b}\rangle \cong C_{2}\times C_{4}$.
		We conclude that \ref{R1} and \ref{E1} hold for $T$. Note that the  explicit presentation of $V$ has not been given before in the literature.
		
			\end{enumerate}
\end{example}
\begin{example}\label{ex:10and14} [\textbf{Groups of orders 10 and 14}] 
The only non-abelian groups of order $10$ and $14$ are dihedral. 
None of these groups is a \cut group, and hence the 
	abelianization of the unit group is not finite. It follows from Corollary~\ref{cor:dihedral} and Theorem~\ref{prop:Dihedral} 
	 that  both \ref{R1} and \ref{E1} have affirmative answers for these groups.
\end{example}

\begin{example}\label{ex:Ord16}
[\textbf{Groups of order 16}] There are nine non-abelian groups of order 16. We are able to  answer \ref{R1} for five of those. For four groups we prove \ref{R1}, while for one group we show that \ref{R2}, and hence also \ref{R1}, does not hold. We prove \ref{E1} for four of the groups, and for the others it remains open. \begin{itemize}
\item[(i)] Three groups for which \ref{R1} and \ref{E1} holds.\\
\begin{itemize}
\item If $G=Q_8\times C_2$ (\texttt{IdSmallGroup(G) = [16,12]}), then as mentioned in the beginning of the section, the unit group consists solely of trivial units and $V/V' \cong C_2\times C_2 \times C_2$. Clearly, (R1) and (E1) hold, in this case.
\item  Let $G=D_8\times C_2$ (\texttt{IdSmallGroup(G)=[16,11]}) and $V = \V(\mathbb{Z}G)$. By \cite[Theorem 3]{Jes95}, the group $V$ is generated by bicyclic units and trivial units. If $b(g,h)$ is a non-trivial bicyclic unit, then $h$ must be of order $2$, as elements of order different from $2$ generate a normal subgroup in $G$. Hence, by Proposition~\ref{prop:bicyclic}, we know that the order of $\varphi(b(g,h))$ divides $2$. We conclude that \ref{R1} and \ref{E1} hold for $G$.
\item If $G=P:=\langle a,b\mid a^4=1,b^4=1,a^b=a^{-1}\rangle$ (\texttt{IdSmallGroup(G) = [16,4]}), the unit group 
		 $\mathrm{V}(\mathbb{Z}G)$ was studied in \cite[Theorem 5.1]{JL91}. It is proved that $\mathrm{V}(\mathbb{Z}G) = C \rtimes G$ where $C$ is a free group of rank 9 and an explicit set of generators of $C$ is given.

			Proceeding as for the group $T$, in Example \ref{ex:Order12}, we write down the relations in the group, by computing the conjugates $v_i^a$ and $v_i^b$ for each $1 \leqslant i \leqslant 9$, using \textsf{GAP} as listed below:

\vspace{10pt}
\begin{minipage}{.25\textwidth}
\begin{itemize}
			\item[(1a)] $v_1^a = v_3v_2$
			\item[(1b)] $v_1^b = v_2^{-1}v_3^{-1}$
			\item[(2a)] $v_2^a = v_3v_1$
			\item[(2b)] $v_2^b = v_2^{-1}$
			\item[(3a)] $v_3^a = v_3^{-1}$ 
			\item[(3b)] $v_3^b = v_1^{-1}v_2$
\end{itemize}
\end{minipage}
\begin{minipage}{.33\textwidth}
\begin{itemize}
			\item[(4a)] $v_4^a = v_3v_5^{-1}v_7v_9v_6^{-1}v_1$ 
			\item[(4b)] $v_4^b = v_4^{-1}$
			\item[(5a)] $v_5^a = (v_5^{-1}v_7v_9v_8)^{v_3^{-1}}$
			\item[(5b)] $v_5^b = v_4^{-1}v_8^{-1}v_6^{-1}v_2$
			\item[(6a)] $v_6^a = (v_2v_4^{-1}v_9^{-1}v_7^{-1}v_5)^{v_3^{-1}}$
			\item[(6b)] $v_6^b = v_2^{-1}v_5^{-1}v_8v_4$
\end{itemize}
\end{minipage}
\begin{minipage}{.33\textwidth}
\begin{itemize}
			\item[(7a)] $v_7^a = (v_9v_8)^{v_7^{-1}v_5v_3^{-1}}$
			\item[(7b)] $v_7^b = (v_8^{-1}v_9^{-1})^{v_4}$
			\item[(8a)] $v_8^a = (v_7v_9)^{v_5v_3^{-1}}$
			\item[(8b)] $v_8^b = (v_8^{-1})^{v_4}$
			\item[(9a)] $v_9^a = (v_9^{-1})^{v_7^{-1}v_5v_3^{-1}}$
			\item[(9b)] $v_9^b = (v_7^{-1}v_8)^{v_4}$
\end{itemize}
\end{minipage}
\vspace{10pt}
	
	From (2b) and (3a) we have that $o(\overline{v_2})$ and $o(\overline{v_3})$ divide $2$ 
	and from (1a) we then also get that $o(\overline{v_1})$ divides $2$. Also from (1a) we have $v_1v_2v_3 \in \ker(\varphi)$, so  $\varphi(\langle v_1, v_2, v_3 \rangle)$ embeds in a $C_2 \times C_2$. Similarly, it follows from (7a)-(9b), that
	$\varphi(\langle v_7, v_8, v_9 \rangle)$ embeds in a $C_2 \times C_2$. 
	Furthermore, by (4b) we have that $o(\overline{v_4})$ divides $2$ and as $v_7v_8v_9 \in \ker(\varphi)$, therefore by (5a) we also have that $o(\overline{v_5})$ divides 2. Consequently, (5b) implies that $o(\overline{v_6})$ divides 2, so that $\varphi(\langle v_4, v_5, v_6 \rangle)$ embeds in a $C_2 \times C_2 \times C_2$. Finally, observing from (5b), $v_2v_4v_5v_6v_8 \in \ker(\varphi)$, equivalently, $\overline{v_4v_5v_6}=\overline{v_2v_8}$, and checking all relations, we obtain $\varphi(C) \cong C_2^6$. 
	Since, $G' = \langle a^2 \rangle$ and $G/G' \cong C_4 \times C_2$, it follows that $V/V'=\varphi(\mathrm{V}(\mathbb{Z}G)) \cong C_4 \times C_2^7$. In particular, \ref{R1} and \ref{E1} hold for $P$. \\
	
	Observe that the relations of the group $G$ and the relations listed in (1a)-(9b), yield a complete set of defining relations of the unit group $V$, which has also not been known before. Using this presentation of $V$, one can also use \textsf{GAP} to conclude $V/V' \cong C_4 \times C_2^7$. \\
	
    Note: There was a mistake in \cite[Theorem 5.1]{JL91}, as the $9^{th}$ generator given there is not a unit in $\mathbb{Z}G$, the origin of the mistake being probably a calculation mistake in the solution of certain linear equations. However, checking the calculations in the last paragraph of the proof, we could rectify this mistake and check that the other generators are given correctly. In fact,
		\[v_9 = 1 +  (2 -20a -9b + 18ab)(1-a^2)(1+b^2).\]
\end{itemize}

\item[(ii)]\label{ex:D16+} A group for which \ref{R2} does not hold, but \ref{E1} does.

If $G=D_{16}^+:=\langle a,b\mid  a^8=b^2=1,a^b=a^{5}\rangle$ (\texttt{IdSmallGroup(G) = [16,6]}), the unit group 
		$\mathrm{V}(\mathbb{Z}G)$ has been  studied in \cite{JL91}, \cite{JP93} 
		and \cite{PdRR05}.
It is proved that $D_{16}^+$ has a torsion-free normal complement $K$ in the unit group and a presentation of $K$ is given in \cite[Theorem 5.3]{PdRR05}. The explicit action of the group base on $K$ is not given there, but can be computed by \cite[Proposition 3.1]{JL91},  via $a =\left( \begin{smallmatrix} 0 & -i \\ -1 & 0 \end{smallmatrix}\right)$ and $b = \left( \begin{smallmatrix} 1 & 0 \\ 0 & -1 \end{smallmatrix}\right)$.

 Generators of the complement are called $A_0$, $A_1$, $P_j$, $Q_j$, $R_j$ and $S_j$ for $0 \leqslant j \leqslant 3$. We list the full action of the group base on $K$, obtained using \textsf{GAP}.\\
		
		 \vspace{10pt}
\begin{minipage}{.29\textwidth}
\begin{itemize}
		 	\item[($A_0$a)] $A_0^a = P_1^{-1}A_0P_3$  
		 	\item[($A_0$b)] $A_0^b = A_0^{-1}$
		 	\item[($A_1$a)] $A_1^a = P_0^{-1}A_1^{-1}P_2$
		 	\item[($A_1$b)] $A_1^b = A_1^{-1}$
		 	\item[($P_0$a)] $P_0^a = P_2^{-1}$  
		 	\item[($P_0$b)] $P_0^b = P_2$
		 	\item[($Q_0$a)] $Q_0^a = P_2^{-1}Q_3^{-1}P_3$
		 	\item[($Q_0$b)] $Q_0^b = Q_2$
		 	\item[($R_0$a)] $R_0^a = P_0^{-1}A_1^{-1}Q_3^{-1}S_3$
		 	\item[($R_0$b)] $R_0^b = R_2$
		 	\item[($S_0$a)] $S_0^a = P_2^{-1}S_3^{-1}P_3$
		 	\item[($S_0$b)] $S_0^b = S_2$

\end{itemize}
\end{minipage}
\begin{minipage}{.29\textwidth}
\begin{itemize}
		 	\item[($P_1$a)] $P_1^a = P_1^{-1}$
		 	\item[($P_1$b)] $P_1^b = P_3$
		 	\item[($Q_1$a)] $Q_1^a = P_1^{-1}Q_2^{-1}P_2$
		 	\item[($Q_1$b)] $Q_1^b = Q_3$
		 	\item[($R_1$a)] $R_1^a = P_0^{-1}Q_1^{-1}A_0S_2$
		 	\item[($R_1$b)] $R_1^b = R_3$
		 	\item[($S_1$a)] $S_1^a = P_1^{-1}S_2^{-1}P_2$
		 	\item[($S_1$b)] $S_1^b = S_3$
		 	\item[($P_2$a)] $P_2^a = P_0^{-1}$
		 	\item[($P_2$b)] $P_2^b = P_0$
		 	\item[($Q_2$a)] $Q_2^a = P_0^{-1}Q_1^{-1}P_1$
		 	\item[($Q_2$b)] $Q_2^b = Q_0$
\end{itemize}
\end{minipage}
\begin{minipage}{.33\textwidth}
\begin{itemize}
		 	\item[($R_2$a)] $R_2^a = P_2^{-1}A_1Q_1^{-1}S_1$
		 	\item[($R_2$b)] $R_2^b = R_0$
		 	\item[($S_2$a)] $S_2^a = P_0^{-1}S_1^{-1}P_1$
		 	\item[($S_2$b)] $S_2^b = S_0$
		 	\item[($P_3$a)] $P_3^a = P_3^{-1}$
		 	\item[($P_3$b)] $P_3^b = P_1$
		 	\item[($Q_3$a)] $Q_3^a = P_3^{-1}Q_0^{-1}P_0$
		 	\item[($Q_3$b)] $Q_3^b = Q_1$
		 	\item[($R_3$a)] $R_3^a = P_1^{-1}A_0Q_0^{-1}S_0$
		 	\item[($R_3$b)] $R_3^b = R_1$
		 	\item[($S_3$a)] $S_3^a = P_3^{-1}S_0^{-1}P_0$
		 	\item[($S_3$b)] $S_3^b = S_1$	
\end{itemize}
\end{minipage}
\vspace{10pt}
		 
Again, as in previous case, the relations of the group and the relations listed above, yield a complete set of defining relations of the unit group $V:=\V(\mathbb{Z}D_{16}^+)$. This presentation has also not been known before. 

A way to easily verify that $V/V'$ is infinite is the following: Define a map $\kappa\colon V = K \rtimes G \to W = \langle x \rangle \cong C_\infty$ by

\begin{align*}
\kappa(a) =  \kappa(b) =  \kappa(A_i) = \kappa(P_i) = \kappa(Q_i) = 1,  \\ 
\kappa(S_j) = \kappa(R_\ell) = x,  \quad \kappa(S_\ell) = \kappa(R_j) = x^{-1},
\end{align*}

where $i \in \{0,1,2,3 \}$, $j \in \{0,2 \}$ and $\ell \in \{1,3 \}$.
Apart from the action of $G$ on $K$, the relations inside of $K$ are the following by \cite[Theorem 5.3]{PdRR05}. Here we use the substitutions $A_2 = A_0^{-1}$ and $A_3 = A_1^{-1}$.

\begin{align*}
\begin{array}{cclclcl}
1 & = & [A_0,A_1] & = & R_1R_2R_3R_0 & = & S_3S_2S_1S_0 \\ 
& = & S_3^{-1}R_2S_1^{-1}R_0 & = & S_0^{-1}R_3S_2^{-1}R_1 &  &  \\ 
& = & P_3^{-1}A_0^{-1}Q_2^{-1}A_0^{-1}Q_1P_0 & = & R_0^{-1}A_1Q_0S_0^{-1}P_0 & = & S_0^{-1}Q_0R_0^{-1}A_1P_0 \\ 
& = & P_0^{-1}A_1^{-1}Q_3^{-1}A_1^{-1}Q_2P_1 & = & R_1^{-1}A_0^{-1}Q_1S_1^{-1}P_1 & = & S_1^{-1}Q_1R_1^{-1}A_0^{-1}P_1 \\ 
& = & P_1^{-1}A_0Q_0^{-1}A_0Q_3P_2 & = & R_2^{-1}A_1^{-1}Q_2S_2^{-1}P_2 & = & S_2^{-1}Q_2R_2^{-1}A_1^{-1}P_2 \\ 
& = & P_2^{-1}A_1Q_1^{-1}A_1Q_0P_3 & = & R_3^{-1}A_0Q_3S_3^{-1}P_3 & = & S_3^{-1}Q_3R_3^{-1}A_0P_3
\end{array} 
\end{align*}

One can now easily check that all these relations are satisfied by the images of $\kappa$ and  $\kappa$ is indeed a homomorphism. Hence $\kappa$ is a surjective group homomorphism from $V$ onto an infinite abelian group. As any abelian image of $V$ is an image of $V/V'$, we conclude that $V/V'$ is infinite. 
On the other hand, $G$ is a \cut group. Consequently, \ref{R2} 
 and hence \ref{R1} does not hold. \\
 
 Using \textsf{GAP}, we check that $V/V' \cong C_\infty \times C_4 \times C_2^5$. Hence \ref{E1} holds for $G$.\\
 
  \noindent Note: In \cite{PdRR05}, there is a typographical error in the definition of the generating elements. Namely, on page 3232 the elements $g_i$ should be defined as $\alpha^ig\alpha^{-i}$ instead of $\alpha^{-i}g\alpha^i$.\\

\item[(iii)] A group for which \ref{R1} holds, but \ref{E1} remains unknown.

If $G=D_{16}$ (\texttt{IdSmallGroup(G) = [16,7]}), the dihedral group of order 16, then \ref{R1} holds by Corollary~\ref{cor:dihedral}.\\

\item[(iv)] The remaining four groups. 
\begin{itemize}
\item If $G=D:=\langle a,b,c\mid  a^2=b^2=c^4=1,a^c=a,b^c=b,a^b=c^2ab\rangle$ (\texttt{IdSmallGroup(G) = [16,13]}), or $G=D_{16}^-:=\langle a,b\mid  a^8=b^2=1,a^b=a^{3}\rangle$ (\texttt{IdSmallGroup(G) = [16,8]}), the unit group $\mathrm{V}(\mathbb{Z}G)$ has also been  studied in \cite{JL91}, \cite{PdRR05} and \cite{PdR06}, and one could, in principle, compute the abelianization of their unit groups, analogous to the case of $D_{16}^+$. 
\item If $G=H:=\langle a,b\mid a^4=b^4=(ab)^2=1,(a^2)^b=a^2\rangle$ (\texttt{IdSmallGroup(G)=[16,3]}), then $G$ is a group of exponent 4 and hence there are no non-trivial Bass units in its integral group ring. But, in view of \cite[Corollary 7]{JP93}, the bicyclic and Bass units do not generate a subgroup of finite index in the full unit group. Though $H$ is a \cut group, but since $\mathbb{Q}H$ has an exceptional component,  we cannot conclude if the abelianization of the unit group for this group is finite or not.	
\item For $G = \langle a,b\mid a^8=1, b^2=a^4,a^b=a^{-1}\rangle$ (\texttt{IdSmallGroup(G)=[16,9]}), the quaternion group of order $16$, exceptional components of both types arise.
We can not decide on \ref{R1} or \ref{E1}. Clearly, the unit group has infinite abelianization, as $G$ it is not a \cut group. 
\end{itemize}

\end{itemize}
		
\end{example}

The following remark implies in particular a proof of Proposition~\ref{theo:FreeComplement}.

\begin{remark}\label{rem:freecomp}
	By \cite{Jes94}, a group $G$ has a normal free complement in the normalized unit group $V = \V(\ZZ G)$, if and only if $G$ is a symmetric group of degree $3$, a dihedral group of order $8$, the dicyclic group $T$ of order $12$ (\texttt{IdSmallGroup(T)=[12,1]}) or the semidirect product $P = C_4 \rtimes C_4$ (\texttt{IdSmallGroup(P)=[16,4]}). Using the computations done above, we observe that for these groups,
	\begin{itemize}
		\item $G = S_3$: $ G/G'\cong C_2$, $V/V' \cong C_2 \times C_2$.
		\item $G = D_8$: $G/G' \cong C_2\times C_2$, $V/V' \cong C_2^4$.
		\item $G = T$: $ G/G' = C_4$, $V/V' \cong C_4 \times C_2 $.
		\item $G = P$: $ G/G' = C_4\times C_2$, $V/V' \cong C_4 \times C_2^7$.
	\end{itemize}
Hence, \ref{E1} and \ref{R1} hold, for the groups $G$, which have a normal free complement in the normalized unit group $\V(\ZZ G)$. 
\end{remark}

We saw in Proposition~\ref{prop:WedderbrunComp} and the comment following it, that question \eqref{var2} has a positive answer for maximal orders in rational group algebras $\QQ G$. Now we will show that the situation for unit groups of the order $\ZZ G$ is more complicated. In doing this, we obtain answers to three questions from \cite{BJJKT1}. As pointed earlier, these type of questions were the original motivation for this article.

\begin{corollary}\label{cor:D16+}
Questions 6.4, 7.8 and 8.3(1) from \cite{BJJKT1} have negative answers, i.e.:
\begin{enumerate}
\item\label{answer6.4} If $\mathcal{O}_1$ and $\mathcal{O}_2$ are two orders in a finite dimensional simple $\mathbb{Q}$-algebra, then the abelianization of $\U(\mathcal{O}_1)$ can be finite though the abelianization of $\U(\mathcal{O}_2)$ is not.
\item\label{answer7.6} If $G$ is a \cut group, the trichotomy statement of fixed point properties of $\U(\mathbb{Z}G)$ described in \cite[Question~7.8]{BJJKT1} is not necessarily true. More precisely, there is a \cut group $G$ such that $\U(\ZZ G)$ neither has the property (HFA), nor the property (FA), nor is a non-trivial amalgamated product with finite abelianization, see \cite{BJJKT1} for the definitions. 
\item\label{answer8.3} There is a group $G$ such that $\U(\ZZ G)$ does not have property (FA), but $\U(\mathcal{O}_i)$ does for every $i$. Here $\QQ G \cong \prod \operatorname{M}_{n_i}(D_i)$ for certain division algebras $D_i$ and $\mathcal{O}_i$ is a maximal order in $\operatorname{M}_{n_i}(D_i)$ for each $i$.
\end{enumerate}
\end{corollary}

\begin{proof}
All proofs are based on Example~\ref{ex:Ord16}(ii),
 i.e., $G = D_{16}^+=\langle a,b\mid  a^8=b^2=1,a^b=a^{5}\rangle$.

 \begin{enumerate}
 	\item We have $\QQ G \cong 4\QQ \times 2\QQ(i) \times \operatorname{M}_2(\QQ(i))$: This follows from the facts that $G' = \langle a^4 \rangle$, $G/G' \cong C_4 \times C_2$ and that $a \mapsto \left( \begin{smallmatrix} 0 & -i \\ -1 & 0\end{smallmatrix} \right), ~b \mapsto \left( \begin{smallmatrix} 1 & 0 \\ 0 & -1 \end{smallmatrix} \right) $  is an irreducible representation of $G$. Clearly, $G$ is a \cut group. Since $\QQ G$ does not have a non-commutative division algebra as simple component and the center of the unit group of a maximal order is finite, the unit group of a maximal $\mathbb{Z}$-order in $\QQ G$ has finite abelianization by Proposition~\ref{prop:WedderbrunComp}. On the other hand, $\ZZ G$ is also an order in $\QQ G$, but $\U(\ZZ G)$ has infinite abelianization by Example~\ref{ex:Ord16}(ii), so we obtain \eqref{answer6.4}.

 	\item It is proved in \cite[Proposition~7.9]{BJJKT1} that a positive answer to \cite[Question~7.8]{BJJKT1} for a \cut group $H$ is equivalent to $\U(\ZZ H)$ having finite abelianization. As this does not hold true for $G$, we obtain a negative answer. Moreover, in all the three properties listed in \cite[Question~7.8]{BJJKT1}, the abelianization of the unit group is finite. Hence, none of those hold for $G$, proving \eqref{answer7.6}. 
 	\item Follows from the fact that, for $G$, in each Wedderburn component of $\QQ G$, a maximal order has either finite unit group or is isomorphic to $\operatorname{M}_2(\ZZ[i])$ and that $\operatorname{GL}_2(\mathbb{Z}[i])$ has property (FA), see \cite[Exercise 5 of Section I.6.5]{Ser80} .\qedhere
 \end{enumerate}
\end{proof}

\section{Remarks}\label{sec:rems}
We provide several relevant observations.

\begin{remark} As mentioned in the introduction, the question whether  Bass and bicyclic units generate a subgroup of finite index in $\V(\ZZ G)$, is of major importance in the study of units in group rings. Here we recover a result of \cite{JP93}, namley that this is not the case for $D_{16}^+$. By Example~\ref{ex:D16+}(ii), we know that \ref{R1} has a negative answer for $D_{16}^+$ and hence Theorem~\ref{theo:ranks} implies that the subgroup of $\V(\ZZ D_{16}^+)$ generated by the Bass and bicyclic units has to be of infinite index in $\V(\ZZ D_{16}^+)$. 
\end{remark}

\begin{remark}
	Let $G= D_{16}^+$ as in Example~\ref{ex:D16+}(ii). In \cite[Theorem 3.1]{DJR07} a (non-explicit) set of units generating $V = \V(\ZZ G)$ was described, given by Bass units and unitary units. Hence, by the result of Propositions~\ref{prop:Bass}, Lemma~\ref{lem:FiniteIndexGeneral} and Example~\ref{ex:D16+}, we conclude that there must be a unitary unit $u$ such that $\varphi(u)$ has infinite order, although $G$ is a \cut group. This is in contrast with the behavior of bicyclic and Bass units, as exhibited in Propositions~\ref{prop:bicyclic} and \ref{prop:Bass}. 
\end{remark}

\begin{remark}
While the question of generating subgroups of finite index in $V = \mathrm{V}(\mathbb{Z}G)$ has found a lot of attention, finite quotients of $V$ and their properties have not been explicitly investigated so far to our knowledge. Our results suggest that at least some classes of finite quotients of $V$ can have interesting properties which are determined by $G$. But it should be noted that a question of type \ref{P} has easy negative solutions for general finite quotients of $V$.

Indeed, let $G = S_3$ and let $a$ be an element of order $3$ in $G$ and $b$ be an involution. Then, as explained in Example~\ref{ex:S3}, $V \cong F \rtimes S_3$, where $F$ is a free group of rank $3$ generated, say by $x$, $y$ and $z$, such that $x^a = y$, $y^a = z$, $z^b = z^{-1}$ and $x^b = y^{-1}$. Let $n$ be any integer and let 
\[N = \langle \langle [x,y], [y,z], [z,x], x^n, y^n, z^n \rangle \rangle_V \]
where $\langle \langle S \rangle \rangle_V$ means the smallest normal subgroup of $V$ containing the elements of $S$. Then $V/N \cong (C_n \times C_n \times C_n) \rtimes S_3$. Hence the exponent of finite quotients of $V$ can be divisible by any integer.

As a more trivial example, one could consider the images of an abelian group, e.g., $\V(\ZZ C_5) \cong C_5 \times C_\infty$ has finite images of arbitrary order.
\end{remark}

\begin{remark} 
	Proposition~\ref{theo:FreeComplement} is not just a consequence of the fact that $V$ is the semidirect product of a free normal subgroup $F$ and a finite group $G$, even when we assume that $G$ acts faithfully on $F$. For instance for $V = F \rtimes C_2$, where $F$ is a free group generated by 2 elements and the $C_2$ acts by interchanging those generators, $\Z(V) =1$ is finite, but $V/V' \cong C_\infty \times C_2$ is not.
\end{remark}

\noindent\textbf{Acknowledgment:}\quad We are thankful to Geoffrey Janssens, Eric Jespers, Ann Kiefer and Doryan Temmerman  
for useful conversations on unit groups and to Martyn Dixon, Joachim Schwermer and Rachel Skipper for insight into infinite groups. We are also thankful to Inder Bir Singh Passi, for pointing out the need to understand the abelianization of the unit group as a first necessary step to understand the lower central series of the unit group of an integral group ring. The second author gratefully acknowledges the support provided by Eric Jespers for stay at Vrije Universiteit Brussel, Belgium, which played a vital role in the outcome of this article. This research was supported in part by the International Centre for Theoretical Sciences (ICTS) during a visit for participating in the program- Group Algebras, Representations and Computation (Code: ICTS/Prog-garc2019/10).

\vspace*{.5cm}

\newcommand{\etalchar}[1]{$^{#1}$}
\providecommand{\bysame}{\leavevmode\hbox to3em{\hrulefill}\thinspace}
\providecommand{\MR}{\relax\ifhmode\unskip\space\fi MR }
\providecommand{\MRhref}[2]{%
  \href{http://www.ams.org/mathscinet-getitem?mr=#1}{#2}
}
\providecommand{\href}[2]{#2}


\begin{thebibliography}{JOdRVG14}

\bibitem[AA69]{AA69}
R.~G. Ayoub and C.~Ayoub, \emph{On the group ring of a finite abelian group},
  Bull. Austral. Math. Soc. \textbf{1} (1969), 245--261.

\bibitem[B{\"a}c18]{Bac18}
A.~B{\"a}chle, \emph{Integral group rings of solvable groups with trivial
  central units}, Forum Math. \textbf{30} (2018), no.~4, 845--855.

\bibitem[BBM20]{BBM20}
V.~Bovdi, T.~Breuer, and A.~Mar\'{o}ti, \emph{Finite simple groups with short
  {G}alois orbits on conjugacy classes}, J. Algebra \textbf{544} (2020),
  151--169.

\bibitem[BCJM18]{BCJM}
A.~B{\"a}chle, M.~Caicedo, E.~Jespers, and S.~Maheshwary, \emph{Global and
  local properties of finite groups with only finitely many central units in
  their integral group ring}, 11 pages, submitted,
  \href{https://arxiv.org/abs/1808.03546v2}{\nolinkurl{arXiv:1808.03546v2
  [math.RA]}}.

\bibitem[BJJ{\etalchar{+}}21]{BJJKT1}
A.~B{\"a}chle, G.~Janssens, E.~Jespers, A.~Kiefer, and D.~Temmerman,
  \emph{Abelianization and fixed point properties of units in integral group
  rings}, Math. Nachr. (2021), to appear,
  \href{https://arxiv.org/abs/1811.12184v4}{\nolinkurl{arXiv:1811.12184v4
  [math.GR]}}.

\bibitem[BMM20]{BMMGit}
A.~B{\"a}chle, S.~Maheswary, and L.~Margolis,
  \emph{Gap-code-for-abelianization-of-unit-groups}, GitHub repository (2020),
  \url{https://github.com/margollo/GAP-code-for-abelianization-of-unit-groups}.

\bibitem[BMP17]{BMP17}
G.~K. Bakshi, S.~Maheshwary, and I.~B.~S. Passi, \emph{Integral group rings
  with all central units trivial}, J. Pure Appl. Algebra \textbf{221} (2017),
  no.~8, 1955--1965.

\bibitem[BZ20]{BZ20}
J.~Belk and M.~Zaremsky, \emph{Twisted {B}rin-{T}hompson groups}, 26 pages,
  \href{https://arxiv.org/abs/2001.04579}{\nolinkurl{arXiv:2001.04579
  [math.GR]}}.

\bibitem[DJ03]{DJ03}
A.~Dooms and E.~Jespers, \emph{Normal complements of the trivial units in the
  unit group of some integral group rings}, Comm. Algebra \textbf{31} (2003),
  no.~1, 475--482.

\bibitem[DJR07]{DJR07}
A.~Dooms, E.~Jespers, and M.~Ruiz, \emph{Free groups and subgroups of finite
  index in the unit group of an integral group ring}, Comm. Algebra \textbf{35}
  (2007), no.~9, 2879--2888.

\bibitem[EKVG15]{EKVG}
F.~Eisele, A.~Kiefer, and I.~Van~Gelder, \emph{Describing units of integral
  group rings up to commensurability}, J. Pure Appl. Algebra \textbf{219}
  (2015), no.~7, 2901--2916.

\bibitem[Eps70]{Eps70}
D.~B.~A. Epstein, \emph{The simplicity of certain groups of homeomorphisms},
  Compositio Math. \textbf{22} (1970), 165--173.

\bibitem[GAP19]{GAP4}
The GAP~Group, \emph{{GAP -- Groups, Algorithms, and Programming, Version
  4.10.2}}, 2019.

\bibitem[Gri20]{Gri20}
N.~Grittini, \emph{A note on cut groups of odd order}, preprint (2020), 3
  pages,
  \href{https://arxiv.org/abs/1911.13196v4}{\nolinkurl{arXiv:1911.13196v4
  [math.GR]}}.

\bibitem[Hig40]{Hig40}
G.~Higman, \emph{The units of group-rings}, Proc. London Math. Soc. (2)
  \textbf{46} (1940), 231--248.

\bibitem[Hyd17]{Hyd17}
J.~Hyde, \emph{Constructing 2-generated subgroups of the group of
  homeomorphisms of cantor space}, 2017, Thesis (Ph.D.)--Univ. St Andrews.

\bibitem[JdR16]{JdR1}
E.~Jespers and \'{A}. del R\'{\i}o, \emph{Group ring groups. {V}ol. 1. {O}rders
  and generic constructions of units}, De Gruyter Graduate, De Gruyter, Berlin,
  2016.

\bibitem[Jes94]{Jes94}
E.~Jespers, \emph{Free normal complements and the unit group of integral group
  rings}, Proc. Amer. Math. Soc. \textbf{122} (1994), no.~1, 59--66.

\bibitem[Jes95]{Jes95}
\bysame, \emph{Bicyclic units in some integral group rings}, Canad. Math. Bull.
  \textbf{38} (1995), no.~1, 80--86.

\bibitem[JL91]{JL91}
E.~Jespers and G.~Leal, \emph{Describing units of integral group rings of some
  {$2$}-groups}, Comm. Algebra \textbf{19} (1991), no.~6, 1809--1827.

\bibitem[JOdRVG14]{JOdRVG}
E.~Jespers, G.~Olteanu, \'{A}. del R\'{\i}o, and I.~Van~Gelder, \emph{Central
  units of integral group rings}, Proc. Amer. Math. Soc. \textbf{142} (2014),
  no.~7, 2193--2209.

\bibitem[JP92]{JP92}
E.~Jespers and M.~M. Parmenter, \emph{Bicyclic units in {${\bf Z}S_3$}}, Bull.
  Soc. Math. Belg. S\'{e}r. B \textbf{44} (1992), no.~2, 141--146.

\bibitem[JP93]{JP93}
\bysame, \emph{Units of group rings of groups of order {$16$}}, Glasgow Math.
  J. \textbf{35} (1993), no.~3, 367--379.

\bibitem[Mah18]{Mah18}
S.~Maheshwary, \emph{Integral {G}roup {R}ings {W}ith {A}ll {C}entral {U}nits
  {T}rivial: {S}olvable {G}roups}, Indian J. Pure Appl. Math. \textbf{49}
  (2018), no.~1, 169--175.

\bibitem[Mah21]{Mah21}
\bysame, \emph{The lower central series of the unit group of an integral group
  ring}, Indian J. Pure Appl. Math. (2021), to appear, 
  \href{https://arxiv.org/abs/2005.12365}{\nolinkurl{arXiv:2005.12365
  [math.GR]}}.

\bibitem[MP18]{MP18}
S.~Maheshwary and I.~B.~S. Passi, \emph{The upper central series of the unit
  groups of integral group rings: a survey}, Group theory and computation,
  Indian Stat. Inst. Ser., Springer, Singapore, 2018, pp.~175--195.

\bibitem[Nek18]{Nek18}
V.~Nekrashevych, \emph{Palindromic subshifts and simple periodic groups of
  intermediate growth}, Ann. of Math. (2) \textbf{187} (2018), no.~3, 667--719.

\bibitem[Neu51]{Neu51}
B.~H. Neumann, \emph{Groups with finite classes of conjugate elements}, Proc.
  London Math. Soc. (3) \textbf{1} (1951), 178--187.

\bibitem[Par93]{Par93}
M.~M. Parmenter, \emph{Free torsion--free normal complements in integral group
  rings}, Comm. Algebra \textbf{21} (1993), no.~10, 3611--3617.

\bibitem[PdR06]{PdR06}
A.~Pita and \'{A}. del R\'{\i}o, \emph{Presentation of the group of units of
  {$\Bbb ZD^-_{16}$}}, Groups, rings and group rings, Lect. Notes Pure Appl.
  Math., vol. 248, Chapman \& Hall/CRC, Boca Raton, FL, 2006, pp.~305--314.

\bibitem[PdRR05]{PdRR05}
A.~Pita, \'{A}. del R\'{\i}o, and M.~Ruiz, \emph{Groups of units of integral
  group rings of {K}leinian type}, Trans. Amer. Math. Soc. \textbf{357} (2005),
  no.~8, 3215--3237.

\bibitem[PS81]{PS81}
D.~S. Passman and P.~F. Smith, \emph{Units in integral group rings}, J. Algebra
  \textbf{69} (1981), no.~1, 213--239.

\bibitem[San81]{San81}
R.~Sandling, \emph{Graham {H}igman's thesis ``{U}nits in group rings''},
  Integral representations and applications ({O}berwolfach, 1980), Lecture
  Notes in Math., vol. 882, Springer, Berlin-New York, 1981, pp.~93--116.

\bibitem[Sch92]{Scheutzow}
A.~Scheutzow, \emph{Computing rational cohomology and {H}ecke eigenvalues for
  {B}ianchi groups}, J. Number Theory \textbf{40} (1992), no.~3, 317--328.

\bibitem[Seh93]{Sehgal1993}
S.~K. Sehgal, \emph{Units in integral group rings}, Pitman Monographs and
  Surveys in Pure and Applied Mathematics, vol.~69, Longman Scientific \&
  Technical, Harlow, 1993.

\bibitem[Ser70]{Ser70}
J.-P. Serre, \emph{Le probl\`eme des groupes de congruence pour {SL}2}, Ann. of
  Math. (2) \textbf{92} (1970), 489--527.

\bibitem[Ser80]{Ser80}
\bysame, \emph{Trees}, Springer-Verlag, Berlin-New York, 1980, Translated from
  the French by John Stillwell.

\bibitem[SG00]{SG00}
R.~K. Sharma and S.~Gangopadhyay, \emph{On chains in units of {${\bf Z}A_4$}},
  Math. Sci. Res. Hot-Line \textbf{4} (2000), no.~9, 1--33.

\bibitem[SG01]{SG01}
\bysame, \emph{On units in {${\bf Z}D_8$}}, PanAmer. Math. J. \textbf{11}
  (2001), no.~1, 1--9.

\bibitem[SGV97]{SGV97}
R.~K. Sharma, S.~Gangopadhyay, and V.~Vetrivel, \emph{On units in {$\bold
  ZS_3$}}, Comm. Algebra \textbf{25} (1997), no.~7, 2285--2299.

\bibitem[Swa71]{Swan}
R.~G. Swan, \emph{Generators and relations for certain special linear groups},
  Advances in Math. \textbf{6} (1971), 1--77 (1971).

\bibitem[SWZ19]{SWZ19}
R.~Skipper, S.~Witzel, and M.~C.~B. Zaremsky, \emph{Simple groups separated by
  finiteness properties}, Invent. Math. \textbf{215} (2019), no.~2, 713--740.

\bibitem[Tre19]{Tre19}
S.~Trefethen, \emph{Non-{A}belian composition factors of finite groups with the
  {CUT}-property}, J. Algebra \textbf{522} (2019), 236--242.

\bibitem[Wei77]{Wei77}
M.~Weinstein, \emph{Examples of groups}, Polygonal Publishing House, Passaic,
  N.J., 1977.

\end{thebibliography}
\end{document}